\documentclass{article}%
\usepackage{amssymb}
\usepackage{amsfonts}
\usepackage{amsmath}
\usepackage{graphicx}%
\providecommand{\U}[1]{\protect\rule{.1in}{.1in}}
\newtheorem{theorem}{Theorem}[section]

\newtheorem{corollary}[theorem]{Corollary}

\newtheorem{definition}[theorem]{Definition}

\newtheorem{lemma}[theorem]{Lemma}

\newtheorem{proposition}[theorem]{Proposition}
\newtheorem{remark}[theorem]{Remark}

\newenvironment{proof}[1][Proof]{\textbf{#1.} }{\hfill\rule{0.5em}{0.5em}}
{\catcode`\@=11\global\let\AddToReset=\@addtoreset
\AddToReset{equation}{section}

\AddToReset{theorem}{section}

\begin{document}

\title{Stability properties for quasilinear parabolic equations with measure data }
\author{Marie-Fran\c{c}oise BIDAUT-VERON\thanks{Laboratoire de Math\'{e}matiques et
Physique Th\'{e}orique, CNRS UMR 7350, Facult\'{e} des Sciences, 37200 Tours
France. E-mail: veronmf@univ-tours.fr}
\and Quoc-Hung NGUYEN\thanks{Laboratoire de Math\'{e}matiques et Physique
Th\'{e}orique, CNRS UMR 7350, Facult\'{e} des Sciences, 37200 Tours France.
E-mail: Hung.Nguyen-Quoc@lmpt.univ-tours.fr}}
\date{.}
\maketitle

\begin{abstract}
Let $\Omega$ be a bounded domain of $\mathbb{R}^{N}$, and $Q=\Omega
\times(0,T).$ We study problems of the model type
\[
\left\{
\begin{array}
[c]{l}%
{u_{t}}-{\Delta_{p}}u=\mu\qquad\text{in }Q,\\
{u}=0\qquad\text{on }\partial\Omega\times(0,T),\\
u(0)=u_{0}\qquad\text{in }\Omega,
\end{array}
\right.
\]
where $p>1$, $\mu\in\mathcal{M}_{b}(Q)$ and $u_{0}\in L^{1}(\Omega).$ Our main
result is a \textit{stability theorem }extending the results of Dal Maso,
Murat, Orsina, Prignet, for the elliptic case, valid for quasilinear operators
$u\longmapsto\mathcal{A}(u)=$div$(A(x,t,\nabla u))$\textit{. }

\end{abstract}

\medskip

\section{Introduction}

Let $\Omega$ be a bounded domain of $\mathbb{R}^{N}$, and $Q=\Omega
\times(0,T),$ $T>0.$ We denote by $\mathcal{M}_{b}(\Omega)$ and $\mathcal{M}%
_{b}(Q)$ the sets of bounded Radon measures on $\Omega$ and $Q$ respectively.
We are concerned with the problem
\begin{equation}
\left\{
\begin{array}
[c]{l}%
{u_{t}}-\text{div}(A(x,t,\nabla u))=\mu\qquad\text{in }Q,\\
{u}=0\qquad\qquad\qquad\qquad\text{on }\partial\Omega\times(0,T),\\
u(0)=u_{0}\qquad\qquad\qquad\text{in }\Omega,
\end{array}
\right.  \label{pmu}%
\end{equation}
where $\mu\in\mathcal{M}_{b}(Q)$, $u_{0}\in L^{1}(\Omega)$ and $A$ is a
Caratheodory function on $Q\times\mathbb{R}^{N}$, such that for $a.e.$
$(x,t)\in Q,$ and any $\xi,\zeta\in\mathbb{R}^{N},$
\begin{equation}
A(x,t,\xi).\xi\geq\Lambda_{1}\left\vert \xi\right\vert ^{p},\qquad\left\vert
A(x,t,\xi)\right\vert \leq a(x,t)+\Lambda_{2}\left\vert \xi\right\vert
^{p-1},\qquad\Lambda_{1},\Lambda_{2}>0,a\in L^{p^{\prime}}(Q),\label{condi1}%
\end{equation}%
\begin{equation}
(A(x,t,\xi)-A(x,t,\zeta)).\left(  \xi-\zeta\right)  >0\qquad\text{ if }\xi
\neq\zeta,\label{condi2}%
\end{equation}
for $p>1.$This includes the model problem where div$(A(x,t,\nabla
u))=\Delta_{p}u,$ where $\Delta_{p}$ is the $p$-Laplacian.\medskip

The corresponding elliptic problem:%
\[
-\Delta_{p}u=\mu\qquad\text{in }\Omega,\qquad u=0\qquad\text{on }%
\partial\Omega,
\]
with $\mu\in\mathcal{M}_{b}(\Omega),$ was studied in \cite{BoGa89,BoGa92} for
$p>2-1/N,$ leading to the existence of solutions in the sense of
distributions. For any $p>1,$ and $\mu\in L^{1}(\Omega),$ existence and
uniqueness are proved in \cite{BBGGPV} in the class of \textit{entropy
solutions}. For any $\mu\in\mathcal{M}_{b}(\Omega)$ the main work is done in
\cite[Theorems 3.1, 3.2]{DMOP}, where not only existence is proved in the
class of \textit{renormalized solutions}, but also a stability result,
fundamental for applications. $\medskip$

Concerning problem (\ref{pmu}), the first studies concern the case $\mu\in
L^{p^{\prime}}(Q)$ and $u_{0}\in L^{2}(\Omega)$, where existence and
uniqueness are obtained by variational methods, see \cite{Li}. In the general
case $\mu\in\mathcal{M}_{b}(Q)$ and $u_{0}\in\mathcal{M}_{b}(\Omega),$ the
pionner results come from \cite{BoGa89}, proving the existence of solutions in
the sense of distributions for%
\begin{equation}
p>p_{1}=2-\frac{1}{N+1},\label{rangep}%
\end{equation}
see also \cite{BDGO97}. The approximated solutions of (\ref{pmu}) lie in
Marcinkiewicz spaces $u\in L^{p_{c},\infty}\left(  Q\right)  $ and $\left\vert
\nabla u\right\vert \in L^{m_{c},\infty}\left(  Q\right)  ,$ where
\begin{equation}
p_{c}=p-1+\frac{p}{N},\qquad m_{c}=p-\frac{N}{N+1}.\label{crit}%
\end{equation}
This condition (\ref{rangep}) ensures that $u$ and $\left\vert \nabla
u\right\vert $ belong to $L^{1}\left(  Q\right)  $, since $m_{c}>1$ means
$p>p_{1}$ and $p_{c}>1$ means $p>2N/(N+1).$ Uniqueness follows in the case
$p=2$, $A(x,t,\nabla u)=\nabla u,$ by duality methods, see \cite{Pe07}.

For $\mu\in L^{1}(Q)$, uniqueness is obtained in new classes of
\textit{entropy solutions}, and \textit{renormalized solutions}, see
\cite{BlMu,Pr97,Xu}.

A larger set of \textit{ }measures is studied in \cite{DrPoPr}. They introduce
a notion of parabolic capacity initiated and inspired by \cite{Pi}, used after
in \cite{Pe08,PePoPor}, defined by
\[
c_{p}^{Q}(E)=\inf(\inf_{E\subset U\text{ open}\subset Q}\{||u||_{W}:u\in
W,u\geq\chi_{U}\quad a.e.\text{ in }Q\}),
\]
for any Borel set $E\subset Q,$ where setting $X={{L^{p}}(}${$(0,T)$}%
${;W_{0}^{1,p}(\Omega)\cap{L^{2}}(\Omega)),}$
\[
W=\left\{  {z:z\in}X{,\quad{z_{t}}\in X}^{\prime}\right\}  ,\text{ embedded
with the norm }||u||_{W}=||u||_{{X}}+||u_{t}||_{{X}^{\prime}}.
\]
Let $\mathcal{M}_{0}(Q)$ be the set of Radon measures $\mu$ on $Q$ that do not
charge the sets of zero $c_{p}^{Q}$-capacity:
\[
\forall E\text{ Borel set }\subset Q,\quad c_{p}^{Q}(E)=0\Longrightarrow
\left\vert \mu\right\vert (E)=0.
\]
Then existence and uniqueness of renormalized solutions of (\ref{pmu}) hold
for any measure $\mu\in\mathcal{M}_{b}(Q)\cap\mathcal{M}_{0}(Q),$ called
\textit{soft (or diffuse, or regular) measure}, and $u_{0}\in L^{1}(\Omega)$,
and $p>1$. The equivalence with the notion of entropy solutions is shown in
\cite{DrPr}. For such a soft measure, an extension to equations of type
$(b({u))_{t}}-\Delta_{p}u=\mu$ is given in \cite{BlPeRe}; another formulation
is used in \cite{PePoPor} for solving a perturbed problem from (\ref{pmu}) by
an absorption term.\medskip

Next consider an\textit{ arbitrary measure} $\mu\in\mathcal{M}_{b}(Q).$ Let
$\mathcal{M}_{s}(Q)$ be the set of all bounded Radon measures on $Q$ with
support on a set of zero $c_{p}^{Q}$-capacity, also called \textit{singular}.
Let $\mathcal{M}_{b}^{+}(Q),\mathcal{M}_{0}^{+}(Q),\mathcal{M}_{s}^{+}(Q)$ be
the positive cones of $\mathcal{M}_{b}(Q),\mathcal{M}_{0}(Q),\mathcal{M}%
_{s}(Q).$ From \cite{DrPoPr}, $\mu$ can be written (in a unique way) under the
form%
\begin{equation}
\mu=\mu_{0}+\mu_{s},\qquad\mu_{0}\in\mathcal{M}_{0}(Q),\quad\mu_{s}=\mu
_{s}^{+}-\mu_{s}^{-},\qquad\mu_{s}^{+},\mu_{s}^{-}\in\mathcal{M}_{s}%
^{+}(Q),\label{deo}%
\end{equation}
and $\mu_{0}\in$ $\mathcal{M}_{0}(Q)$ admits (at least) a decomposition under
the form%
\begin{equation}
\mu_{0}=f-\operatorname{div}g+h_{t},\qquad f\in L^{1}(Q),\quad g\in
(L^{p^{\prime}}(Q))^{N},\quad h\in{X},\label{dec}%
\end{equation}
and we write $\mu_{0}=(f,g,h).$ Conversely, any measure of this form,
\textit{such that} $h\in L^{\infty}(Q),$ lies in $\mathcal{M}_{0}(Q),$ see
\cite[Proposition 3.1]{PePoPor}. The solutions of (\ref{pmu}) are searched in
a renormalized sense linked to this decomposition, introduced in
\cite{DrPoPr,Pe08}. \ In the range (\ref{rangep}) the existence of a
renormalized solution relative to the decomposition (\ref{dec}) is proved in
\cite{Pe08}, using suitable approximations of $\mu_{0}$ and $\mu_{s}$.
Uniqueness is still open, as well as in the elliptic case. $\medskip$

In \textit{all the sequel} we suppose that $p$ satisfies (\ref{rangep}). Then
the embedding $W_{0}^{1,p}(\Omega)\subset L^{2}(\Omega)$ is valid, that means
\[
X={{L^{p}}((0,T);W_{0}^{1,p}(\Omega)),\qquad X}^{\prime}={{L^{p^{\prime}}%
}((0,T);W^{-1,p^{\prime}}(\Omega)).}%
\]

In Section \ref{defsol} we recall the definition of renormalized solutions,
given in \cite{Pe08}, that we call R-solutions of (\ref{pmu}), relative to the
decomposition (\ref{dec}) of $\mu_{0}$, and study some of their properties.
Our main result is a \textit{stability theorem} for problem (\ref{pmu}),
proved in Section \ref{cv}, extending to the parabolic case the stability
result of \cite[Theorem 3.4]{DMOP}. In order to state it, we recall that a
sequence of measures $\mu_{n}\in\mathcal{M}_{b}(Q)$ converges to a measure
$\mu\in\mathcal{M}_{b}(Q)$ in the \textit{narrow topology} of measures if%
\[
\lim_{n\rightarrow\infty}\int_{Q}\varphi d\mu_{n}=\int_{Q}\varphi d\mu
\qquad\forall\varphi\in C(Q)\cap L^{\infty}(Q).
\]

\begin{theorem}
\label{sta} Let $A:Q\times\mathbb{R}^{N}\rightarrow\mathbb{R}^{N}$ satisfy
(\ref{condi1}),(\ref{condi2}). Let $u_{0}\in L^{1}(\Omega)$, and
\[
\mu=f-\operatorname{div}g+h_{t}+\mu_{s}^{+}-\mu_{s}^{-}\in\mathcal{M}_{b}%
({Q}),
\]
with $f\in L^{1}(Q),g\in(L^{p^{\prime}}(Q))^{N},$ $h\in X$ and $\mu_{s}%
^{+},\mu_{s}^{-}\in\mathcal{M}_{s}^{+}(Q).$ Let $u_{0,n}\in L^{1}(\Omega),$
\[
\mu_{n}=f_{n}-\operatorname{div}g_{n}+(h_{n})_{t}+\rho_{n}-\eta_{n}%
\in\mathcal{M}_{b}({Q}),
\]
with \ $f_{n}\in L^{1}(Q),g_{n}\in(L^{p^{\prime}}(Q))^{N},h_{n}\in X,$ and
$\rho_{n},\eta_{n}\in\mathcal{M}_{b}^{+}({Q}),$ such that
\[
\rho_{n}=\rho_{n}^{1}-\operatorname{div}\rho_{n}^{2}+\rho_{n,s},\qquad\eta
_{n}=\eta_{n}^{1}-\mathrm{\operatorname{div}}\eta_{n}^{2}+\eta_{n,s},
\]
with $\rho_{n}^{1},\eta_{n}^{1}\in L^{1}(Q),\rho_{n}^{2},\eta_{n}^{2}%
\in(L^{p^{\prime}}(Q))^{N}$ and $\rho_{n,s},\eta_{n,s}\in\mathcal{M}_{s}%
^{+}(Q).$ Assume that
\[
\sup_{n}\left\vert {{\mu_{n}}}\right\vert ({Q})<\infty,
\]
and $\left\{  u_{0,n}\right\}  $ converges to $u_{0}$ strongly in
$L^{1}(\Omega),$ $\left\{  f_{n}\right\}  $ converges to $f$ weakly in
$L^{1}(Q),$ $\left\{  g_{n}\right\}  $ converges to $g$ strongly in
$(L^{p^{\prime}}(Q))^{N}$, $\left\{  h_{n}\right\}  $ converges to $h$
strongly in $X$, $\left\{  \rho_{n}\right\}  $ converges to $\mu_{s}^{+}$ and
$\left\{  \eta_{n}\right\}  $ converges to $\mu_{s}^{-}$ in the narrow
topology; and $\left\{  \rho_{n}^{1}\right\}  ,\left\{  \eta_{n}^{1}\right\}
$ are bounded in $L^{1}(Q)$, and $\left\{  \rho_{n}^{2}\right\}  ,\left\{
\eta_{n}^{2}\right\}  $ bounded in $(L^{p^{\prime}}(Q))^{N}$.\medskip

Let $\left\{  u_{n}\right\}  $ be a sequence of R-solutions of
\begin{equation}
\left\{
\begin{array}
[c]{l}%
{u_{n,t}}-\text{div}(A(x,t,\nabla u_{n}))=\mu_{n}\qquad\text{in }Q,\\
{u}_{n}=0\qquad\text{on }\partial\Omega\times(0,T),\\
u_{n}(0)=u_{0,n}\qquad\text{in }\Omega.
\end{array}
\right.  \label{pmun}%
\end{equation}
relative to the decomposition $(f_{n}+\rho_{n}^{1}-\eta_{n}^{1},g_{n}+\rho
_{n}^{2}-\eta_{n}^{2},h_{n})$ of $\mu_{n,0}.$ Let $U_{n}=u_{n}-h_{n}.\medskip$

Then up to a subsequence, $\left\{  u_{n}\right\}  $ converges $a.e.$ in $Q$
to a R-solution $u$ of (\ref{pmu}), and $\left\{  U_{n}\right\}  $ converges
$a.e.$ in $Q$ to $U=u-h.$ Moreover, $\left\{  \nabla u_{n}\right\}  ,\left\{
\nabla U_{n}\right\}  $ converge respectively to $\nabla u,\nabla U$ $a.e.$ in
$Q,$ and $\left\{  T_{k}(U_{n})\right\}  $ converge to $T_{k}(U)$ strongly in
$X$ for any $k>0$.\bigskip
\end{theorem}

In Section \ref{prox} we check that any measure $\mu$ $\in\mathcal{M}_{b}%
({Q})$ can be approximated in the sense of the stability Theorem, hence we
find again the existence result of \cite{Pe08}:

\begin{corollary}
\label{051120131} Let $u_{0}\in L^{1}(\Omega)$ and $\mu\in\mathcal{M}_{b}(Q)$.
Then there exists a R-solution $u$ to the problem (\ref{pmu}) with data
$(\mu,u_{0}).$
\end{corollary}

Moreover we give more precise properties of approximations of $\mu
\in\mathcal{M}_{b}(Q),$ fundamental for applications, see Propositions
\ref{4bhatt} and \ref{P5}. As in the elliptic case, Theorem \ref{sta} is a key
point for obtaining existence results for more general problems, and we give
some of them in \cite{BiNgQu1,BiNgQu2,NgQu}, for measures $\mu$ satisfying
suitable capacitary conditions. In \cite{BiNgQu1} we study perturbed problems
of order $0$, of type
\begin{equation}
{u_{t}}-\Delta_{p}u+\mathcal{G}(u)=\mu\qquad\text{in }Q,\label{equg}%
\end{equation}
where $\mathcal{G}(u)$ is an absorption or a source term with a growth of
power or exponential type, and $\mu$ is a good in time measure. In
\cite{BiNgQu2} we use potential estimates to give other existence results in
case of absorption with $p>2$. In \cite{NgQu}, one considers equations of the
form
\[
{u_{t}}-\operatorname{div}(A(x,t,\nabla u))+\mathcal{G}(u,\nabla u)=\mu
\]
under (\ref{condi1}),(\ref{condi2}) with $p=2,$ and extend in particular the
results of \cite{BaPi} to nonlinear operators.

\section{Renormalized solutions of problem (\ref{pmu})\label{defsol}}

\subsection{Notations and Definition}

For any function $f\in L^{1}(Q),$ we write $\int_{Q}f$ instead of $\int%
_{Q}fdxdt,$ and for any measurable set $E\subset${$Q$}${,}$ $\int_{E}f$
instead of $\int_{E}fdxdt.$ For any open set $\varpi$ of $\mathbb{R}^{m}$ and
$F\in(L^{k}(\varpi))^{\nu},$ $k\in\left[  1,\infty\right]  ,m,\nu\in
\mathbb{N}^{\ast},$ we set $\left\Vert F\right\Vert _{k,\varpi}=\left\Vert
F\right\Vert _{(L^{k}(\varpi))^{\nu}}$\medskip

\noindent We set $T_{k}(r)=\max\{\min\{r,k\},-k\},$ for any $k>0$ and
$r\in\mathbb{R}$. We recall that if $u$ is a measurable function defined and
finite $a.e.$ in $Q$, such that $T_{k}(u)\in X$ for any $k>0$, there exists a
measurable function $w$ from $Q$ into $\mathbb{R}^{N}$ such that $\nabla
T_{k}(u)=\chi_{|u|\leq k}w,$ $a.e.$ in $Q,$ and for any $k>0$. We define the
gradient $\nabla u$ of $u$ by $w=\nabla u$. \medskip

\noindent Let $\mu=\mu_{0}+\mu_{s}\in\mathcal{M}_{b}(${$Q$}$),$ and $(f,g,h)$
be a decomposition of $\mu_{0}$ given by (\ref{dec}), and ${{\widehat{\mu_{0}%
}}}=\mu_{0}-h_{t}=f-\operatorname{div}g$. In the general case ${{\widehat{\mu
_{0}}}}\notin\mathcal{M}(Q),$ but we write, for convenience,%
\[
\int_{Q}{wd{\widehat{\mu_{0}}}}:=\int_{Q}(fw+g.\nabla w),\qquad\forall w\in
X{\cap}L^{\infty}(Q).
\]

\begin{definition}
\label{defin}Let ${u}_{0}\in L^{1}(\Omega),$ $\mu=\mu_{0}+\mu_{s}%
\in\mathcal{M}_{b}(${$Q$}$)$. A measurable function $u$ is a
\textbf{renormalized solution, }called\textbf{\ R-solution} of (\ref{pmu}) if
there exists a decompostion $(f,g,h)$ of $\mu_{0}$ such that
\begin{equation}
U=u-h\in L^{\sigma}((0,T);W_{0}^{1,\sigma}(\Omega))\cap L^{\infty}%
((0,T);L^{1}(\Omega)),\quad\forall\sigma\in\left[  1,m_{c}\right)  ;\qquad
T_{k}(U)\in X,\quad\forall k>0, \label{defv}%
\end{equation}
and:\medskip

(i) for any $S\in W^{2,\infty}(\mathbb{R})$ such that $S^{\prime}$ has compact
support on $\mathbb{R}$, and $S(0)=0$,%
\begin{equation}
-\int_{\Omega}S(u_{0})\varphi(0)dx-\int_{Q}{{\varphi_{t}}S(U)}+\int%
_{Q}{S^{\prime}(U)A(x,t,\nabla u).\nabla\varphi}+\int_{Q}{S^{\prime\prime
}(U)\varphi A(x,t,\nabla u).\nabla U}=\int_{Q}{S^{\prime}(U)\varphi
d{\widehat{\mu_{0}},}} \label{renor}%
\end{equation}
for any $\varphi\in X\cap L^{\infty}(Q)$ such that $\varphi_{t}\in X^{\prime
}+L^{1}(Q)$ and $\varphi(.,T)=0$;\medskip

(ii) for any $\phi\in C(\overline{{Q}}),$%
\begin{equation}
\lim_{m\rightarrow\infty}\frac{1}{m}\int\limits_{\left\{  m\leq U<2m\right\}
}{\phi A(x,t,\nabla u).\nabla U}=\int_{Q}\phi d\mu_{s}^{+} \label{renor2}%
\end{equation}%
\begin{equation}
\lim_{m\rightarrow\infty}\frac{1}{m}\int\limits_{\left\{  -m\geq
U>-2m\right\}  }{\phi A(x,t,\nabla u).\nabla U}=\int_{Q}\phi d\mu_{s}^{-}.
\label{renor3}%
\end{equation}

\end{definition}

\begin{remark}
As a consequence, $S(U)\in C([0,T];L^{1}(\Omega))$ and ${S(}${$U$%
}${)(.,0)=S(u}_{0})$ in $\Omega;$ and $u$ satisfies the equation
\begin{equation}
({S(U))}_{t}-\operatorname{div}({S^{\prime}(U)A(x,t,\nabla u))+S^{\prime
\prime}(U)A(x,t,\nabla u).\nabla U{=f}S^{\prime}(U)-\operatorname{div}%
(gS^{\prime}(U))+S^{\prime\prime}(U)g.\nabla U,}\text{ } \label{dpri}%
\end{equation}
in the sense of distributions in $Q,$ see \cite[Remark 3]{Pe08}. Moreover
assume that $\left[  -k,k\right]  \supset$ supp$S^{\prime}.$ then from
(\ref{condi1}) and the H\"{o}lder inequality, we find easily that
\begin{align}
{\left\Vert {S{{(U)}_{t}}}\right\Vert _{{X}^{\prime}+{L^{1}}({Q})}}  &  \leq
C\left\Vert S\right\Vert _{{W^{2,\infty}}(\mathbb{R})}\left(  {}\right.
\left\Vert {\left\vert {\nabla u}\right\vert }^{p}\chi_{|U|\leq k}\right\Vert
_{1,Q}^{1/p^{\prime}}+\left\Vert {|\nabla u|^{p}\chi_{|U|\leq k}}\right\Vert
_{1,Q}+\left\Vert |\nabla T_{k}(U)|\right\Vert _{p,Q}^{p}\nonumber\\
&  +\left\Vert a\right\Vert _{p^{\prime},Q}+\left\Vert a\right\Vert
_{p^{\prime},Q}^{p^{\prime}}+\left\Vert f\right\Vert _{1,Q}+\left\Vert
g\right\Vert _{p^{\prime},Q}\left\Vert \left\vert {\nabla u}\right\vert
^{p}\chi_{|U|\leq k}\right\Vert _{1,Q}^{1/p}+\left\Vert g\right\Vert
_{p^{\prime},Q}\left.  {}\right)  , \label{11051}%
\end{align}
where $C=C(p,\Lambda_{2}).$ We also deduce that, for any $\varphi\in X\cap
L^{\infty}(Q),$ such that $\varphi_{t}{\in X}^{\prime}+L^{1}(Q),$
\begin{align}
\int_{\Omega}S(U(T))\varphi(T)dx-\int_{\Omega}S(u_{0})\varphi(0)dx  &
-\int_{Q}{{\varphi_{t}}S(U)}+\int_{Q}{S^{\prime}(U)A(x,t,\nabla u).\nabla
\varphi}\nonumber\\
&  +\int_{Q}{S^{\prime\prime}(U)A(x,t,\nabla u).\nabla U\varphi}=\int%
_{Q}{S^{\prime}(U)\varphi d{\widehat{\mu_{0}}.}} \label{renor4}%
\end{align}

\end{remark}

\begin{remark}
Let $u,U$ satisfy (\ref{defin}). It is easy to see that the condition
(\ref{renor2}) ( resp. (\ref{renor3}) ) is equivalent to
\begin{equation}
\lim_{m\rightarrow\infty}\frac{1}{m}\int\limits_{\left\{  m\leq U<2m\right\}
}{\phi A(x,t,\nabla u).\nabla u}=\int_{Q}\phi d\mu_{s}^{+} \label{lim1}%
\end{equation}
resp.
\begin{equation}
\lim_{m\rightarrow\infty}\frac{1}{m}\int\limits_{\left\{  m\geq U>-2m\right\}
}{\phi A(x,t,\nabla u).\nabla u}=\int_{Q}\phi d\mu_{s}^{-}. \label{lim2}%
\end{equation}
In particular, for any $\varphi\in L^{p^{\prime}}(Q)$ there holds
\begin{equation}
\lim_{m\rightarrow\infty}\frac{1}{m}\int\limits_{m\leq|U|<2m}{|\nabla
u|\varphi}=0,\qquad\lim_{m\rightarrow\infty}\frac{1}{m}\int\limits_{m\leq
|U|<2m}{|\nabla U|\varphi}=0. \label{lim3}%
\end{equation}

\end{remark}

\begin{remark}
(i) Any function $U\in X$ such that $U_{t}\in X^{\prime}+L^{1}(Q)$ admits a
unique $c_{p}^{Q}$-quasi continuous representative, defined $c_{p}^{Q}$-quasi
$a.e.$ in $Q,$ still denoted $U.$ Furthermore, if $U\in L^{\infty}(Q),$ then
for any $\mu_{0}\in\mathcal{M}_{0}(Q),$ there holds $U\in L^{\infty}%
(Q,d\mu_{0}),$ see \cite[Theorem 3 and Corollary 1]{Pe08}.\medskip

(ii) Let $u$ be any R- solution of problem (\ref{pmu}). Then, $U=u-h$ admits a
$c_{p}^{Q}$-quasi continuous functions representative which is finite
$c_{p}^{Q}$-quasi $a.e.$ in $Q,$ and $u$ satisfies definition \ref{defin} for
every decomposition $(\tilde{f},\tilde{g},\tilde{h})$ such that $h-\tilde
{h}\in L^{\infty}(Q)$, see \cite[Proposition 3 and Theorem 4 ]{Pe08}.
\end{remark}

\subsection{Steklov and Landes approximations}

\textit{A main difficulty for proving Theorem \ref{sta} is the choice of
admissible test functions }$(S,\varphi)$\textit{ in (\ref{renor}), valid for
any R-solution}. Because of a lack of regularity of these solutions, we use
two ways of approximation adapted to parabolic equations:

\begin{definition}
\label{ste}Let $\varepsilon\in(0,T)$ and $z\in L_{loc}^{1}(Q)$. For any
$l\in(0,\varepsilon)$ we define the \textbf{Steklov time-averages}
$[z]_{l},[z]_{-l}$ of $z$ by
\[
{[z]_{l}}(x,t)=\frac{1}{l}\int\limits_{t}^{t+l}{z(x,s)ds}\qquad\text{for
}a.e.\;(x,t)\in\Omega\times(0,T-\varepsilon),
\]%
\[
{[z]_{-l}}(x,t)=\frac{1}{l}\int\limits_{t-l}^{t}{z(x,s)ds}\qquad\text{for
}a.e.\;(x,t)\in\Omega\times(\varepsilon,T).
\]

\end{definition}

\noindent The idea to use this approximation for R-solutions can be found in
\cite{BlPo}. Recall some properties, given in \cite{PePoPor}. Let
$\varepsilon\in(0,T),$ and $\varphi_{1}\in C_{c}^{\infty}(\overline{\Omega
}\times\lbrack0,T)),\;\varphi_{2}\in C_{c}^{\infty}(\overline{\Omega}%
\times(0,T])$ with $\mathrm{Supp}\varphi{{_{1}}}\subset\overline{\Omega}%
\times\lbrack0,T-\varepsilon],\;\mathrm{Supp}\varphi{{_{2}}}\subset
\overline{\Omega}\times\lbrack\varepsilon,T]$. There holds: \medskip

\noindent(i) If $z\in{X}$, then $\varphi_{1}{[z]_{l}}$ and $\varphi
_{2}{[z]_{-l}}\in{W.}$\medskip

\noindent(ii) If $z\in X$ and $z_{t}\in X^{\prime}+L^{1}(Q),$ then, as
$l\rightarrow0,$ $(\varphi_{1}{[z]_{l})}$ and $(\varphi_{2}{[z]_{-l})}$
converge respectively to $\varphi_{1}z$ and $\varphi_{2}z$ in $X,$ and $a.e.$
in $Q;$ and $(\varphi_{1}{[z]_{l})}_{t},(\varphi_{2}{[z]_{-l})}_{t}$ converge
to $(\varphi_{1}z)_{t},(\varphi_{2}z)_{t}$ in $X^{\prime}+L^{1}(Q)$.\medskip

\noindent(iii) If moreover $z\in L^{\infty}(Q)$, then from any sequence
$\{l_{n}\}\rightarrow0,$ there exists a subsequence $\{l_{\nu}\}$ such that
$\left\{  [z]_{l_{\nu}}\right\}  ,\left\{  [z]_{-l_{\nu}}\right\}  $ converge
to $z,$ $c_{p}^{Q}$-quasi everywhere in $Q.$\medskip

Next we recall the approximation used in several articles
\cite{BlPo1,DAOr,BDGO97}, first introduced in \cite{La}.

\begin{definition}
Let $k>0$, and $y\in L^{\infty}(\Omega)$ and $Y\in X$ such that
$||y||_{L^{\infty}(\Omega)}\leq k$ and $||Y||_{L^{\infty}(Q)}\leq k$. For any
$\nu\in\mathbb{N},$ a \textbf{Landes-time approximation} $\langle
Y\rangle_{\nu}$ of the function $Y$ is defined as follows:
\[
\langle Y\rangle_{\nu}(x,t)=\nu\int_{0}^{t}Y(x,s)e^{\nu(s-t)}ds+e^{-\nu
t}z_{\nu}(x),\qquad\forall(x,t)\in Q.
\]
where $\left\{  z_{\nu}\right\}  $ is a sequence of functions in $W_{0}%
^{1,p}(\Omega)\cap L^{\infty}(\Omega)$, such that $||z_{\nu}||_{L^{\infty
}(\Omega)}\leq k$, $\left\{  z_{\nu}\right\}  $ converges to $y$ $a.e.$ in
$\Omega$, and $\nu^{-1}||z_{\nu}||_{W_{0}^{1,p}(\Omega)}^{p}$ converges to $0$.
\end{definition}

Therefore, we can verify that $(\langle Y\rangle_{\nu})_{t}\in X$, $\langle
Y\rangle_{\nu}\in X\cap L^{\infty}(Q)$, $||\langle Y\rangle_{\nu}||_{\infty
,Q}\leq k$ and $\left\{  \langle Y\rangle_{\nu}\right\}  $ converges to $Y$
strongly in $X$ and $a.e.$ in $Q$. Moreover, $\langle Y\rangle_{\nu}$
satisfies the equation $(\langle Y\rangle_{\nu})_{t}=\nu\left(  Y-\langle
Y\rangle_{\nu}\right)  $ in the sense of distributions in $Q,$ and $\langle
Y\rangle_{\nu}(0)={z_{\nu}}\text{ in }\Omega$. In this paper, we only use the
\textbf{Landes-time approximation} of the function $Y=T_{k}(U),$ where
$y=T_{k}(u_{0})$.

\subsection{First properties}

In the sequel we use the following notations: for any function $J\in
{W^{1,\infty}}(\mathbb{R})$, nondecreasing with $J(0)=0,$ we set
\begin{equation}
\overline{J}(r)=\int_{0}^{r}J{(\tau)d\tau,}\text{\qquad}\mathcal{J}%
(r)=\int_{0}^{r}J{^{\prime}(\tau)\tau d\tau.} \label{lam}%
\end{equation}
It is easy to verify that $\mathcal{J}(r)\geq0,$
\begin{equation}
\mathcal{J}(r)+\overline{J}(r)=J(r)r,\quad\text{and\quad}\mathcal{J}%
(r)-\mathcal{J}(s)\geq s\left(  J{(}r{)-J(}s{)}\right)  \quad\quad\forall
r,s\in\mathbb{R}. \label{222}%
\end{equation}
In particular we define, for any $k>0,$ and any $r\in\mathbb{R},$
\begin{equation}
\overline{T_{k}}(r)=\int_{0}^{r}T_{k}{(\tau)d\tau,}\text{\qquad}%
\mathcal{T}_{k}(r)=\int_{0}^{r}T_{k}^{\prime}{(\tau)\tau d\tau,} \label{tkp}%
\end{equation}
and we use several times a truncature used in \cite{DMOP}:
\begin{equation}
H{_{m}}(r)={\chi_{\lbrack-m,m]}}(r)+\frac{{2m-|s|}}{m}{\chi_{m<|s|\leq2m}%
}(r),\qquad\overline{H_{m}}(r)=\int_{0}^{r}H{_{m}(\tau)d\tau.} \label{Hm}%
\end{equation}

The next Lemma allows to extend the range of the test functions in
(\ref{renor}).

\begin{lemma}
\label{integ}Let $u$ be a R-solution of problem (\ref{pmu}). Let
$J\in{W^{1,\infty}}(\mathbb{R})$ be nondecreasing with $J(0)=0$, and
$\overline{J}$ defined by (\ref{lam}). Then,%
\begin{align}
&  \int_{Q}{{S^{\prime}(U)A(x,t,\nabla u).\nabla\left(  {\xi J(S(U))}\right)
}}+\int_{Q}{S^{\prime\prime}(U)A(x,t,\nabla u).\nabla U\xi J(S(U))}\nonumber\\
&  {-\int_{\Omega}{{\xi(0)J(S({u_{0}}))S({u_{0}})dx}}-}\int_{Q}{{{\xi_{t}%
}\overline{J}(S(U))}}\leq\int_{Q}{{S^{\prime}(U)\xi J(S(U))d\widehat{\mu_{0}%
}{,}}}\label{partu}%
\end{align}
for any $S\in{W^{2,\infty}}(\mathbb{R})$ such that $S^{\prime}$ has compact
support on $\mathbb{R}$ and $S(0)=0,$ and for any $\xi\in C^{1}(Q)\cap
W^{1,\infty}(Q),\xi\geq0.$
\end{lemma}

\begin{proof}
Let $\mathcal{J}{\ }$be defined by (\ref{lam}). Let $\zeta\in C_{c}%
^{1}([0,T))$ with values in $[0,1],$ such that $\zeta_{t}\leq0$, and
$\varphi=\zeta\xi\lbrack j{(S(}${$U$}${))}]_{l}$. Clearly, $\varphi\in X\cap
L^{\infty}(Q)$; we choose the pair of functions $(\varphi,S)$ as test function
in (\ref{renor}). From the convergence properties of Steklov time-averages, we
easily will obtain (\ref{partu}) if we prove that
\[
\lim_{\overline{l\rightarrow0,\zeta\rightarrow1}}(-\int_{Q}{{{\left(
{\zeta\xi{{\left[  {j(S(U))}\right]  }_{l}}}\right)  }_{t}}S(U))}\geq-\int%
_{Q}{{\xi_{t}}\overline{J}(S(U)).}%
\]
We can write $-\int_{Q}{{{\left(  {\zeta\xi{{\left[  {j(S(U))}\right]  }_{l}}%
}\right)  }_{t}}S(}${$U$}${)}=F+G,$ with
\[
F=-\int_{Q}{(\zeta\xi)_{t}}{{\left[  {j(S(U))}\right]  }_{l}}S(U),\qquad
G=-\int_{Q}{\zeta\xi S(U)\frac{1}{l}\left(  {j(S(U))(x,t+l)-j(S(U))(x,t)}%
\right)  .}%
\]
\newline Using (\ref{222}) and integrating by parts we have
\begin{align*}
G &  \geq-\int_{Q}{\zeta\xi\frac{1}{l}\left(  \mathcal{J}{(S(U))(x,t+l)-}%
\mathcal{J}{(S(U))(x,t)}\right)  }=-\int_{Q}{\zeta\xi\frac{\partial}{{\partial
t}}\left(  {{{\left[  \mathcal{J}{(S(U))}\right]  }_{l}}}\right)  }\\
&  =\int_{Q}{{(\zeta\xi)_{t}}{{\left[  \mathcal{J}{(S(U))}\right]  }_{l}}%
}+\int_{\Omega}{\zeta(0)\xi(0)}{\left[  \mathcal{J}{(S(U))}\right]  _{l}%
}(0)dx\geq\int_{Q}{{(\zeta\xi)_{t}}{{\left[  \mathcal{J}{(S(U))}\right]  }%
_{l}}},
\end{align*}
since $\mathcal{J}{(S(}${$U$}${))}$ $\geq0.$ Hence,
\[
-\int_{Q}{{{\left(  {\zeta\xi{{\left[  {j(S(U))}\right]  }_{l}}}\right)  }%
_{t}}S(U)}\geq\int_{Q}{{(\zeta\xi)_{t}}{{\left[  \mathcal{J}{(S(U))}\right]
}_{l}}}+F=\int_{Q}{{(\zeta\xi)_{t}}\left(  {{{\left[  \mathcal{J}%
{(S(U))}\right]  }_{l}}-{{\left[  J{(S(U))}\right]  }_{l}}S(U)}\right)  .}%
\]
Otherwise, $\mathcal{J}(S(U))$ and $J(S(U))\in C(\left[  0,T\right]  {;L}%
^{1}{(\Omega))}$, thus $\left\{  {(\zeta\xi)_{t}}\left(  {{{\left[
\mathcal{J}{(S(u))}\right]  }_{l}}-{{\left[  J{(S(u))}\right]  }_{l}}%
S(u)}\right)  \right\}  $ converges to $-{(\zeta\xi)_{t}}\overline{J}(S(u))$
in $L^{1}(${$Q$}$)$ as $l\rightarrow0$. Therefore,
\[
\lim_{\overline{l\rightarrow0,\zeta\rightarrow1}}({-}\int_{Q}{{{{\left(
{\zeta\xi{{[J(S(U))]}_{l}}}\right)  }_{t}}S(U))}}\geq\lim_{\overline
{\zeta\rightarrow1}}\left(  {-}\int_{Q}{{{{\left(  {\zeta\xi}\right)  }_{t}%
}\overline{J}(S(U))}}\right)  \geq-\int_{Q}{{\xi_{t}}\overline{J}(S(U))},
\]
which achieves the proof.\medskip
\end{proof}

Next we give estimates of the function and its gradient, following the first
ones of \cite{BDGO97}, inspired by the estimates of the elliptic case of
\cite{BBGGPV}. In particular we extend and make more precise the a priori
estimates of \cite[Proposition 4]{Pe08} given for solutions with smooth data;
see also \cite{DrPoPr,LePe}.

\begin{proposition}
\label{estsup}If $u$ is a R-solution of problem (\ref{pmu}), then there exists
$C_{1}=C_{1}(p,\Lambda_{1},\Lambda_{2})$ such that, for any $k\geq1$ and
$\ell$ $\geq0,$
\begin{equation}
\int\limits_{\ell\leq|U|\leq\ell+k}{{{\left\vert {\nabla u}\right\vert }^{p}%
+}}\int\limits_{{\ell}\leq{|U|}\leq{\ell+k}}{{{\left\vert {\nabla
U}\right\vert }^{p}}}\leq C_{1}kM, \label{alp}%
\end{equation}%
\begin{equation}
{\left\Vert U\right\Vert _{{L^{\infty}}(((0,T));{L^{1}}(\Omega))}}\leq
C_{1}({M}+|\Omega|), \label{gam}%
\end{equation}
where $M={{{\left\Vert {{u_{0}}}\right\Vert }_{1,\Omega}}}+\left\vert {\mu
_{s}}\right\vert (Q){+}\left\Vert f\right\Vert _{1,Q}+\left\Vert g\right\Vert
_{p^{\prime},Q}^{p^{\prime}}+{\left\Vert h\right\Vert _{{X}}^{p}%
}+||a||_{p^{\prime},Q}^{p^{\prime}}.$

\noindent As a consequence, for any $k$ $\geq1,$
\begin{equation}
\mathrm{meas}\left\{  {|U|>k}\right\}  \leq{C_{2}M_{1}}{k^{-p_{c}},\qquad
}\mathrm{meas}\left\{  {|\nabla U|>k}\right\}  \leq{C_{2}M_{2}}{k^{-m_{c}},}
\label{mess}%
\end{equation}%
\begin{equation}
\mathrm{meas}\left\{  {|u|>k}\right\}  \leq{C_{2}M_{2}}{k^{-p_{c}},\qquad
}\mathrm{meas}\left\{  {|\nabla u|>k}\right\}  \leq{C_{2}M_{2}}{k^{-m_{c}},}
\label{mess2}%
\end{equation}
where $C_{2}=C_{2}(N,p,\Lambda_{1},\Lambda_{2}),$ and ${M_{1}}={\left(
M{+|\Omega|}\right)  ^{\frac{p}{N}}}M$ and ${M_{2}}=M_{1}+M.$
\end{proposition}

\begin{proof}
Set for any $r\in\mathbb{R}$, and $m,k,\ell>0,$
\[
T_{k,\ell}(r)=\max\{\min\{r-\ell,k\},0\}+\min\{\max\{r+\ell,-k\},0\}.
\]
For $m>k+\ell$, we can choose $(J,S,\xi)=(T_{k,\ell},\overline{H_{m}},\xi)$ as
test functions in (\ref{partu}), where $\overline{H_{m}}$ is defined at
(\ref{Hm}) and $\xi\in C^{1}([0,T])$ with values in $[0,1]$, independent on
$x$. Since $T_{k,\ell}(\overline{H_{m}}(r))=T_{k,\ell}(r)$ for all
$r\in\mathbb{R},$ we obtain
\[%
\begin{array}
[c]{l}%
-\int_{\Omega}{\xi(0){T_{k,\ell}}({u_{0}})\overline{H_{m}}({u_{0}})dx}%
-\int_{Q}{{\xi_{t}}}\overline{T_{k,\ell}}{(\overline{H_{m}}(U))}\\
\\
+\int\limits_{\left\{  {\ell}\leq{|U|}<{\ell+k}\right\}  }{\xi A(x,t,\nabla
u).\nabla U}-\frac{k}{m}\int\limits_{\left\{  m\leq|U|<2m\right\}  }{\xi
A(x,t,\nabla u).\nabla U}\leq\int_{Q}H_{m}{(U)\xi{T_{k,\ell}}%
(U)d{{\widehat{\mu_{0}}}.}}%
\end{array}
\]
And
\[
\int_{Q}H_{m}{(U)\xi{T_{k,\ell}}(U)d{{\widehat{\mu_{0}}}}\;}{=}\int_{Q}%
H_{m}{(U)\xi{T_{k,\ell}}(U)}f{{{+}}}\int\limits_{\left\{  {\ell}\leq
{|U|}<{\ell+k}\right\}  }{\xi\nabla U.g-}\frac{k}{m}\int\limits_{\left\{
m\leq|U|<2m\right\}  }{\xi\nabla U.g.}%
\]
Let $m\rightarrow\infty$; then, for any $k\geq1,$ since $U\in L^{1}(Q)$ and
from (\ref{renor2}), (\ref{renor3}), and (\ref{lim3}), we find%
\begin{equation}
-\int_{Q}{{\xi_{t}}\overline{T_{k,\ell}}(U)}+\int\limits_{\left\{  {\ell}%
\leq{|U|}<{\ell+k}\right\}  }{\xi A(x,t,\nabla u).\nabla U\;\leq}%
\int\limits_{\left\{  {\ell}\leq{|U|}<{\ell+k}\right\}  }{\xi\nabla
U.g}+k({{{\left\Vert {{u_{0}}}\right\Vert }_{1,\Omega}+}}\left\vert {\mu_{s}%
}\right\vert (Q){+}\left\Vert f\right\Vert _{1,Q}).\label{epsa}%
\end{equation}
Next, we take $\xi\equiv1$. We verify that
\[
{A(x,t,\nabla u).\nabla U-\nabla U.g\geq}\frac{\Lambda_{1}}{4}(|\nabla
u|^{p}+|\nabla U|^{p})-c_{1}(\left\vert g\right\vert ^{p^{\prime}}+|\nabla
h|^{p}+|a|^{p^{\prime}})
\]
for some $c_{1}=c_{1}(p,\Lambda_{1},\Lambda_{2})>0$. Hence (\ref{alp})
follows. Thus, from (\ref{epsa}) and the H\"{o}lder inequality, we get, for
any $\xi\in C^{1}([0,T])$ with values in $[0,1],$
\[
-\int_{Q}{{\xi_{t}}\overline{T_{k,\ell}}(U)}\leq c_{2}kM
\]
for some $c_{2}=c_{2}(p,\Lambda_{1},\Lambda_{2})>0.$Thus $\int_{\Omega
}\overline{T_{k,\ell}}{(}${$U$}${)(t)dx}\leq c_{2}kM,$ for $a.e.$ $t\in(0,T).$
We deduce (\ref{gam}) by taking $k=1,\ell=0$, since $\overline{T_{1,0}%
}(r)=\overline{T_{1}}(r)$ $\geq|r|-1,$ for any $r\in\mathbb{R}.$ \medskip

\noindent Next, from the Gagliardo-Nirenberg embedding Theorem, see
\cite[Proposition 3.1]{DiBe}, we have
\[
\int_{Q}{{{\left\vert {{T_{k}}(U)}\right\vert }^{\frac{{p(N+1)}}{N}}}}\leq
c_{3}\left\Vert U\right\Vert _{{L^{\infty}}(((0,T));{L^{1}}(\Omega))}%
^{\frac{p}{N}}\int_{Q}{{{\left\vert {\nabla{T_{k}}(U)}\right\vert }^{p}},}%
\]
where $c_{3}=c_{3}(N,p).$ Then, from (\ref{alp}) and (\ref{gam}), we get, for
any $k$ $\geq1,$
\[
\mathrm{meas}\left\{  {|U|>k}\right\}  \leq{k^{-\frac{{p(N+1)}}{N}}}\int%
_{Q}{{{\left\vert {{T_{k}}(U)}\right\vert }^{\frac{{p(N+1)}}{N}}}}\leq
c_{3}\left\Vert U\right\Vert _{{L^{\infty}}((0,T);{L^{1}}(\Omega))}^{\frac
{p}{N}}{k^{-\frac{{p(N+1)}}{N}}}\int_{Q}{{{\left\vert {\nabla{T_{k}}%
(U)}\right\vert }^{p}}}\leq c_{4}M_{1}{k^{-p_{c}}},\text{ }%
\]
with $c_{4}=c_{4}(N,p,\Lambda_{1},\Lambda_{2})$. We obtain%
\begin{align*}
\mathrm{meas}\left\{  {|\nabla U|>k}\right\}   &  \leq\frac{1}{{{k^{p}}}}%
\int_{0}^{{k^{p}}}{\mathrm{{meas}}}\left(  {\left\{  {|\nabla U{|^{p}}%
>s}\right\}  }\right)  ds\\
&  \leq\mathrm{{meas}}\left\{  {|U|>{k^{\frac{N}{{N+1}}}}}\right\}  +\frac
{1}{{{k^{p}}}}\int_{0}^{{k^{p}}}{\mathrm{{meas}}}\left(  {\left\{  {|\nabla
U{|^{p}}>s,|U|}\leq{{k^{\frac{N}{{N+1}}}}}\right\}  }\right)  ds\\
&  \leq c_{4}{M_{1}}{k^{-m_{c}}}+\frac{1}{{{k^{p}}}}\int\limits_{|U|\leq
{k^{\frac{N}{{N+1}}}}}{{{\left\vert {\nabla U}\right\vert }^{p}}}\leq
c_{5}{M_{2}}{k^{-m_{c}}},
\end{align*}
with $c_{5}=c_{5}(N,p,\Lambda_{1},\Lambda_{2}).$ Furthermore, for any
$k\geq1,$
\[
\mathrm{meas}\left\{  {|h|>k}\right\}  +\mathrm{meas}\left\{  {|\nabla
h|>k}\right\}  \leq c_{6}k^{-p}\left\Vert h\right\Vert _{X}^{p},
\]
where $c_{6}=c_{6}(N,p)$. Therefore, we easily get (\ref{mess2}).
\end{proof}

\begin{remark}
\label{h2} If $\mu\in L^{1}(Q)$ and $a\equiv0$ in (\ref{condi1}), then
(\ref{alp}) holds for all $k>0$ and the term $|\Omega|$ in inequality
(\ref{gam}) can be removed, where $M=||u_{0}||_{1,\Omega}+|\mu|(Q)$.
Furthermore, (\ref{mess2}) is stated as follows:
\begin{equation}
\mathrm{meas}\left\{  {|u|>k}\right\}  \leq C_{2}M^{\frac{p+N}{N}}{k^{-p_{c}%
},\qquad}\mathrm{meas}\left\{  {|\nabla u|>k}\right\}  \leq C_{2}M^{\frac
{N+2}{N+1}}k^{-m_{c}},\forall k>0.\label{1810131}%
\end{equation}
with $C_{2}=C_{2}(N,p,\Lambda_{1},\Lambda_{2}).$To see last inequality, we do
in the following way:
\begin{align*}
\mathrm{meas}\left\{  {|\nabla U|>k}\right\}   &  \leq\mathrm{{meas}}\left\{
{|U|>M^{\frac{1}{N+1}}k^{\frac{N}{{N+1}}}}\right\}  +\frac{1}{{{k^{p}}}}%
\int_{0}^{{k^{p}}}{\mathrm{{meas}}\left\{  |\nabla U{|^{p}}>s,|U|\leq
M^{\frac{1}{N+1}}k^{\frac{N}{{N+1}}}\right\}  }ds\\
&  \leq C_{2}M^{\frac{N+2}{N+1}}k^{-m_{c}}.
\end{align*}

\end{remark}

\begin{proposition}
\label{mun} Let $\{\mu_{n}\}$ $\subset$ $\mathcal{M}_{b}(${$Q$}$),$ and
$\{u_{0,n}\}\subset L^{1}(\Omega),$ such that
\[
\sup_{n}\left\vert {{\mu_{n}}}\right\vert ({Q})<\infty,\text{ and }\sup
_{n}||{{u_{0,n}}}||_{1,\Omega}<\infty.
\]
Let $u_{n}$ be a R-solution of (\ref{pmu}) with data $\mu_{n}=\mu_{n,0}%
+\mu_{n,s}$ and $u_{0,n},$ relative to a decomposition $(f_{n},g_{n},h_{n})$
of $\mu_{n,0}$, and $U_{n}=u_{n}-h_{n}.$ Assume that $\{f_{n}\}$ is bounded in
$L^{1}(Q)$, $\{g_{n}\}$ bounded in $(L^{p^{\prime}}(Q))^{N}$ and $\{h_{n}\}$
bounded in $X$. \medskip

\noindent Then, up to a subsequence, $\{U_{n}\}$ converges $a.e.$ to a
function $U\in{L^{\infty}}((0,T);{L^{1}}(\Omega)),$ such that $T_{k}(U)\in X$
for any $k>0$and $U\in L^{\sigma}((0,T);W_{0}^{1,\sigma}(\Omega))$ for any
$\sigma\in\lbrack1,m_{c}).$ And\medskip

\noindent{(i)} $\left\{  U_{n}\right\}  $ converges to $U$ strongly in
$L^{\sigma}(Q)$ for any $\sigma\in\lbrack1,m_{c}),$ and $\sup{\left\Vert
{{U_{n}}}\right\Vert _{{L^{\infty}}((0,T);{L^{1}}(\Omega))}}<\infty,$\medskip

\noindent{(ii)} $\sup_{k>0}\sup_{n}\frac{1}{k+1}\int_{Q}|\nabla T_{k}%
(U_{n})|^{p}<\infty$,\medskip

\noindent{(iii)} $\left\{  T_{k}(U_{n})\right\}  $ converges to $T_{k}(U)$
weakly in $X,$ for any $k>0$,\medskip

\noindent{(iv)} $\left\{  A\left(  x,t,\nabla\left(  T_{k}(U_{n}%
)+h_{n}\right)  \right)  \right\}  $ converges to some $F_{k}$ weakly in
$(L^{p^{\prime}}(Q))^{N}$. \medskip
\end{proposition}

\begin{proof}
Take $S\in W^{2,\infty}(\mathbb{R})$ such that $S^{\prime}$ has compact
support on $\mathbb{R}$ and $S(0)=0$. We combine (\ref{11051}) with
(\ref{alp}), and deduce that $\{S(U_{n})_{t}\}$ is bounded in ${X}^{\prime
}+{L^{1}}(${$Q$}$)$ and $\{S(U_{n})\}$ bounded in $X$. Hence, $\{S(U_{n})\}$
is relatively compact in $L^{1}(Q)$. On the other hand, we choose $S=S_{k}$
such that $S_{k}(z)=z,$ if $\left\vert z\right\vert <k$ and $S(z)=2k\;$%
sign$z,$ if $|z|>2k.$ From (\ref{gam}), we obtain%
\begin{align}
\mathrm{meas}\left\{  {\left\vert {{U_{n}}-{U_{m}}}\right\vert >\sigma
}\right\}   &  \leq\mathrm{meas}\left\{  {\left\vert {{U_{n}}}\right\vert
>k}\right\}  +\mathrm{meas}\left\{  {\left\vert {{U_{m}}}\right\vert
>k}\right\}  +\mathrm{meas}\left\{  {\left\vert {{S_{k}}({U_{n}})-{S_{k}%
}({U_{m}})}\right\vert >\sigma}\right\}  \nonumber\\
&  \leq\frac{c}{k}+\mathrm{meas}\left\{  {\left\vert {{S_{k}}({U_{n}})-{S_{k}%
}({U_{m}})}\right\vert >\sigma}\right\}  ,\nonumber
\end{align}
where $c$ does not depend of $n,m.$ Thus, up to a subsequence $\{u_{n}\}$ is a
Cauchy sequence in measure, and converges $a.e.$ in $Q$ to a function $u$.
Thus, $\left\{  T_{k}(U_{n})\right\}  $ converges to $T_{k}(U)$ weakly in $X$,
since $\sup_{n}{\left\Vert {{T_{k}}({U_{n}})}\right\Vert _{X}}<\infty$ for any
$k>0$. And $\left\{  |\nabla\left(  T_{k}(U_{n})+h_{n}\right)  |^{p-2}%
\nabla\left(  T_{k}(U_{n})+h_{n}\right)  \right\}  $ converges to some $F_{k}$
weakly in $(L^{p^{\prime}}(Q))^{N}$. Furthermore, from (\ref{mess}), $\left\{
U_{n}\right\}  $ strongly converges to $U$ in $L^{\sigma}(Q),$ for any
$\sigma<p_{c}.$
\end{proof}

\section{The convergence theorem\label{cv}}

We first recall some properties of the measures, see \cite[Lemma 5]{Pe08},
\cite{DMOP}.

\begin{proposition}
\label{04041} Let $\mu_{s}=\mu_{s}^{+}-\mu_{s}^{-}\in\mathcal{M}_{b}(Q),$
where $\mu_{s}^{+}$ and $\mu_{s}^{-}$ are concentrated, respectively, on two
disjoint sets $E^{+}$ and $E^{-}$ of zero $c_{p}^{Q}$-capacity. Then, for any
$\delta>0$, there exist two compact sets $K_{\delta}^{+}\subseteq E^{+}$ and
$K_{\delta}^{-}\subseteq E^{-}$ such that
\[
\mu_{s}^{+}(E^{+}\backslash K_{\delta}^{+})\leq\delta,\text{\qquad}\mu_{s}%
^{-}(E^{-}\backslash K_{\delta}^{-})\leq\delta,
\]
and there exist $\psi_{\delta}^{+},\psi_{\delta}^{-}\in C_{c}^{1}(Q)$ with
values in $\left[  0,1\right]  ,$ such that $\psi_{\delta}^{+},\psi_{\delta
}^{-}=1$ respectively on $K_{\delta}^{+},K_{\delta}^{-},$ and $\text{supp}%
(\psi_{\delta}^{+})\cap\text{supp}(\psi_{\delta}^{-})=\emptyset$, and
\[
||\psi_{\delta}^{+}||_{X}+||(\psi_{\delta}^{+})_{t}||_{X^{\prime}+L^{1}%
(Q)}\leq\delta,\qquad||\psi_{\delta}^{-}||_{X}+||(\psi_{\delta}^{-}%
)_{t}||_{X^{\prime}+L^{1}(Q)}\leq\delta.
\]
There exist decompositions $(\psi_{\delta}^{+})_{t}={\left(  {\psi_{\delta
}^{+}}\right)  _{t}^{1}+\left(  {\psi_{\delta}^{+}}\right)  _{t}^{2}}$ and
$(\psi_{\delta}^{-})_{t}={\left(  {\psi_{\delta}^{-}}\right)  _{t}^{1}+\left(
{\psi_{\delta}^{-}}\right)  _{t}^{2}}$ in $X^{\prime}+L^{1}(Q),$ such that
\begin{equation}
{\left\Vert {\left(  {\psi_{\delta}^{+}}\right)  _{t}^{1}}\right\Vert
_{{X}^{\prime}}}\leq\frac{\delta}{3},\qquad{\left\Vert {\left(  {\psi_{\delta
}^{+}}\right)  _{t}^{2}}\right\Vert _{1,Q}}\leq\frac{\delta}{3},\qquad
{\left\Vert {\left(  {\psi_{\delta}^{-}}\right)  _{t}^{1}}\right\Vert
_{{X}^{\prime}}}\leq\frac{\delta}{3},\qquad{\left\Vert {\left(  {\psi_{\delta
}^{-}}\right)  _{t}^{2}}\right\Vert _{1,Q}}\leq\frac{\delta}{3}.\label{41}%
\end{equation}
Both $\left\{  \psi_{\delta}^{+}\right\}  $ and $\left\{  \psi_{\delta}%
^{-}\right\}  $ converge to $0$, weak-$^{\ast}$ in $L^{\infty}(Q)$, and
strongly in $L^{1}(Q)$ and up to subsequences, $a.e.$ in $Q,$ as $\delta$
tends to $0$.

\noindent Moreover if $\rho_{n}$ and $\eta_{n}$ are as in Theorem \ref{sta},
we have, for any $\delta,\delta_{1},\delta_{2}>0,$
\begin{equation}
\int_{Q}{\psi_{\delta}^{-}}d\rho_{n}+\int_{Q}{\psi_{\delta}^{+}}d\eta
_{n}=\omega(n,\delta),\qquad\int_{Q}{\psi_{\delta}^{-}}d\mu_{s}^{+}\leq
\delta,\qquad\int_{Q}{\psi_{\delta}^{+}}d\mu_{s}^{-}\leq\delta, \label{12054}%
\end{equation}%
\begin{equation}
\int_{Q}{(1-\psi_{\delta_{1}}^{+}\psi_{\delta_{2}}^{+})}d\rho_{n}%
=\omega(n,\delta_{1},\delta_{2}),\qquad\int_{Q}{(1-\psi_{\delta_{1}}^{+}%
\psi_{\delta_{2}}^{+})}d\mu_{s}^{+}\leq\delta_{1}+\delta_{2}, \label{12056}%
\end{equation}%
\begin{equation}
\int_{Q}{(1-\psi_{\delta_{1}}^{-}\psi_{\delta_{2}}^{-})}d\eta_{n}%
=\omega(n,\delta_{1},\delta_{2}),\qquad\int_{Q}{(1-\psi_{\delta_{1}}^{-}%
\psi_{\delta_{2}}^{-})}d\mu_{s}^{-}\leq\delta_{1}+\delta_{2}. \label{12057}%
\end{equation}

\end{proposition}

Hereafter, if $n,\varepsilon,...,\nu$ are real numbers, and a function $\phi$
depends on $n,\varepsilon,...,\nu$ and eventual other parameters $\alpha
,\beta,..,\gamma$, and $n\rightarrow n_{0},\varepsilon\rightarrow
\varepsilon_{0},..,$ $\nu\rightarrow\nu_{0}$, we write $\phi=\omega
(n,\varepsilon,..,\nu)$, then this means that, for fixed $\alpha
,\beta,..,\gamma,$ there holds\textbf{ }$\overline{\lim}_{\nu\rightarrow
{\nu_{0}}}..\overline{\lim}_{\varepsilon\rightarrow{\varepsilon_{0}}}%
\overline{\lim}_{n\rightarrow{n_{0}}}\left\vert \phi\right\vert =0$\textbf{.}
In the same way, $\phi\leq\omega(n,\varepsilon,\delta,...,\nu)$ means
$\overline{\lim}_{\nu\rightarrow{\nu_{0}}}..\overline{\lim}_{\varepsilon
\rightarrow{\varepsilon_{0}}}\overline{\lim}_{n\rightarrow{n_{0}}}\phi\leq0$,
and $\phi$ $\geq\omega(n,\varepsilon,..,\nu)$ means $-\phi\leq\omega
(n,\varepsilon,..,\nu).$

\begin{remark}
\label{05041}In the sequel we recall a convergence property  still used in
\cite{DMOP}: If $\left\{  b_{1,n}\right\}  $ is a sequence in $L^{1}(Q)$
converging to $b_{1}$ weakly in $L^{1}(Q)$ and $\left\{  b_{2,n}\right\}  $ a
bounded sequence in $L^{\infty}(Q)$ converging to $b_{2},$ $a.e.$ in $Q,$ then
$\lim_{n\rightarrow\infty}\int_{Q}{{b_{1,n}b_{2,n}}}=\int_{Q}{{b_{1}b_{2}}.}$
\end{remark}

Next we prove Thorem \ref{sta}.

\begin{proof}
[Scheme of the proof]Let $\{\mu_{n}\},\left\{  u_{0,n}\right\}  $ and
$\left\{  u_{n}\right\}  $ satisfy the assumptions of Theorem \ref{sta}. Then
we can apply Proposition \ref{mun}. Setting $U_{n}=u_{n}-h_{n},$ up to
subsequences, $\left\{  u_{n}\right\}  $ converges $a.e.$ in $Q$ to some
function $u,$ and $\left\{  U_{n}\right\}  $ converges $a.e.$ to $U=u-h,$ such
that $T_{k}(U)\in X$ for any $k>0,$ and $U\in L^{\sigma}((0,T);W_{0}%
^{1,\sigma}(\Omega))\cap{L^{\infty}}((0,T);{L^{1}}(\Omega))$ for every
$\sigma\in\left[  1,m_{c}\right)  $. And $\{U_{n}\}$ satisfies the conclusions
(i) to (iv) of Proposition \ref{mun}. We have
\begin{align*}
\mu_{n} &  =(f_{n}-\operatorname{div}g_{n}+(h_{n})_{t})+(\rho_{n}%
^{1}-\operatorname{div}\rho_{n}^{2})-(\eta_{n}^{1}-\operatorname{div}\eta
_{n}^{2})+\rho_{n,s}-\eta_{n,s}\\
&  =\mu_{n,0}+(\rho_{n,s}-\eta_{n,s})^{+}-(\rho_{n,s}-\eta_{n,s})^{-},
\end{align*}
where
\begin{equation}
\mu_{n,0}=\lambda_{n,0}+\rho_{n,0}-\eta_{n,0},\text{ \quad with }\lambda
_{n,0}=f_{n}-\operatorname{div}g_{n}+(h_{n})_{t},\quad\rho_{n,0}=\rho_{n}%
^{1}-\operatorname{div}\rho_{n}^{2},\quad\eta_{n,0}=\eta_{n}^{1}%
-\operatorname{div}\eta_{n}^{2}.\label{muni}%
\end{equation}
Hence
\begin{equation}
\rho_{n,0},\eta_{n,0}\in\mathcal{M}_{b}^{+}(Q)\cap\mathcal{M}_{0}%
(Q),\text{\quad and\quad}\rho_{n}\geq\rho_{n,0},\quad\eta_{n}\geq\eta
_{n,0}.\label{muno}%
\end{equation}

\noindent Let $E^{+},E^{-}$ be the sets where, respectively, $\mu_{s}^{+}$ and
$\mu_{s}^{-}$ are concentrated. For any $\delta_{1},\delta_{2}>0$, let
$\psi_{\delta_{1}}^{+},\psi_{\delta_{2}}^{+}$ and $\psi_{\delta_{1}}^{-}%
,\psi_{\delta_{2}}^{-}$ as in Proposition \ref{04041} and set
\[
\Phi_{\delta_{1},\delta_{2}}=\psi_{\delta_{1}}^{+}\psi_{\delta_{2}}^{+}%
+\psi_{\delta_{1}}^{-}\psi_{\delta_{2}}^{-}.
\]
\textit{Suppose that we can prove the two estimates, near }$E$
\begin{equation}
I_{1}:=\int\limits_{\left\{  |U_{n}|\leq k\right\}  }{\Phi_{\delta_{1}%
,\delta_{2}}A(x,t,\nabla u_{n}).\nabla\left(  {{U_{n}}-}\langle T_{k}%
(U)\rangle_{\nu}\right)  }\leq\omega(n,\nu,\delta_{1},\delta_{2}),
\label{12059}%
\end{equation}
\textbf{ }\textit{and far from} $E,$%
\begin{equation}
I_{2}:=\int\limits_{\left\{  |U_{n}|\leq k\right\}  }{(1-\Phi_{\delta
_{1},\delta_{2}})A(x,t,\nabla u_{n}).\nabla({{U_{n}}-}\langle T_{k}%
(U)\rangle_{\nu})}\leq\omega(n,\nu,\delta_{1},\delta_{2}). \label{120510}%
\end{equation}
Then it follows that
\begin{equation}
\overline{\lim}_{n,\nu}\int\limits_{\left\{  |U_{n}|\leq k\right\}
}{A(x,t,\nabla u_{n}).\nabla\left(  {{U_{n}}-}\langle{{{{{T_{k}}(U)\rangle}%
}_{\nu}}}\right)  }\leq0, \label{12052}%
\end{equation}
which implies%
\begin{equation}
\overline{\lim}_{n\rightarrow\infty}\int\limits_{\left\{  |U_{n}|\leq
k\right\}  }{A(x,t,\nabla u_{n}).\nabla\left(  {{U_{n}}-{T_{k}}(U)}\right)
}\leq0, \label{12061}%
\end{equation}
since $\left\{  \langle{{{{{T_{k}}(U)\rangle}}_{\nu}}}\right\}  $ converges to
$T_{k}(U)$ in $X.$ On the other hand, from the weak convergence of $\left\{
T_{k}(U_{n})\right\}  $ to $T_{k}(U)$ in $X,$ we verify that%
\[
\int\limits_{\left\{  |U_{n}|\leq k\right\}  }{A(x,t,\nabla(T_{k}%
(U)+h_{n})).\nabla\left(  {{T_{k}}({U_{n}})-{T_{k}}(U)}\right)  }=\omega(n).
\]
Thus we get
\[
\int\limits_{\left\{  |U_{n}|\leq k\right\}  }{\left(  {A(x,t,\nabla
u_{n})-A(x,t,\nabla(T_{k}(U)+h_{n}))}\right)  .\nabla\left(  {{u_{n}}-\left(
{{T_{k}}(U)+{h_{n}}}\right)  }\right)  }=\omega(n).
\]
Then, it is easy to show that, up to a subsequence,
\begin{equation}
\left\{  \nabla u_{n}\right\}  \text{ converges to }\nabla u,\qquad\text{
}a.e.\text{ in }Q. \label{pp}%
\end{equation}
Therefore, $\left\{  A(x,t,\nabla u_{n})\right\}  $ converges to $A(x,t,\nabla
u)$ weakly in $(L^{p^{\prime}}(Q))^{N}$ ; and from (\ref{12061}) we find
\[
\overline{\lim}_{n\rightarrow\infty}\int_{Q}A(x,t,\nabla u_{n}).\nabla{T_{k}%
}({U_{n}})\leq\int_{Q}A(x,t,\nabla u)\nabla T_{k}(U).
\]
Otherwise, $\left\{  {A(x,t,\nabla\left(  {{T_{k}}(U_{n})+{h_{n}}}\right)
)}\right\}  $ converges weakly in $(L^{p^{\prime}}(Q))^{N}$to some $F_{k},$
from Proposition \ref{mun}, and we obtain that $F_{k}={A(x,t,\nabla\left(
{{T_{k}}(U)+{h}}\right)  ).}$ Hence
\begin{align*}
&  \overline{\lim}_{n\rightarrow\infty}\int_{Q}A(x,t,\nabla(T_{k}(U_{n}%
)+h_{n})).\nabla(T_{k}(U_{n})+h_{n})\\
&  \leq\overline{\lim}_{n\rightarrow\infty}\int_{Q}A(x,t,\nabla u_{n}).\nabla
T_{k}(U_{n})+\overline{\lim}_{n\rightarrow\infty}\int_{Q}A(x,t,\nabla
(T_{k}(U_{n})+h_{n})).\nabla h_{n}\\
&  \leq\int_{Q}A(x,t,\nabla(T_{k}(U)+h)).\nabla(T_{k}(U)+h).
\end{align*}
As a consequence
\begin{equation}
\left\{  T_{k}(U_{n})\right\}  \text{ converges to }T_{k}(U),\text{ strongly
in }X,\qquad\forall k>0. \label{02041}%
\end{equation}
Then \textit{to finish the proof we have to check that }$u$\textit{ is a
solution of} (\ref{pmu}).\medskip\medskip
\end{proof}

In order to prove (\ref{12059}) we need a first Lemma, inspired of \cite[Lemma
6.1]{DMOP}. It extends the results of \cite[Lemma 6 and Lemma 7]{Pe08}
relative to sequences of solutions with smooth data:

\begin{lemma}
\label{april261}Let $\psi_{1,\delta},\psi_{2,\delta}\in C^{1}(Q)$ be uniformly
bounded in $W^{1,\infty}(Q)$ with values in $[0,1],$ and such that $\int%
_{Q}{\psi_{1,\delta}}d\mu_{s}^{-}\leq\delta$ and $\int_{Q}{\psi_{2,\delta}%
}d\mu_{s}^{+}\leq\delta$. Let $\left\{  u_{n}\right\}  $ satisfying the
assumptions of Theorem \ref{sta}, and $U_{n}=u_{n}-h_{n}.$  Then%
\begin{equation}
\frac{1}{m}\int\limits_{\left\{  m\leq{U_{n}}<2m\right\}  }{{{\left\vert
{\nabla{u_{n}}}\right\vert }^{p}}{\psi_{2,\delta}}}=\omega(n,m,\delta
),\quad\quad\frac{1}{m}\int\limits_{\left\{  m\leq{U_{n}}<2m\right\}
}{{{\left\vert {\nabla{U_{n}}}\right\vert }^{p}}{\psi_{2,\delta}}}%
=\omega(n,m,\delta),\label{13051}%
\end{equation}%
\begin{equation}
\frac{1}{m}\int\limits_{-2m<{U_{n}}\leq-m}{{{\left\vert {\nabla{u_{n}}%
}\right\vert }^{p}}{\psi_{1,\delta}}}=\omega(n,m,\delta),\qquad\frac{1}{m}%
\int\limits_{-2m<{U_{n}}\leq-m}{{{\left\vert {\nabla{U_{n}}}\right\vert }^{p}%
}{\psi_{1,\delta}}}=\omega(n,m,\delta),\label{13052}%
\end{equation}
and for any $k>0,$%
\begin{equation}
\int\limits_{\left\{  m\leq{U_{n}}<m+k\right\}  }{{{\left\vert {\nabla{u_{n}}%
}\right\vert }^{p}}{\psi_{2,\delta}}}=\omega(n,m,\delta),\qquad\int%
\limits_{\left\{  m\leq{U_{n}}<m+k\right\}  }{{{\left\vert {\nabla{U_{n}}%
}\right\vert }^{p}}{\psi_{2,\delta}}}=\omega(n,m,\delta),\label{13053}%
\end{equation}%
\begin{equation}
\int\limits_{\left\{  -m-k<{U_{n}}\leq-m\right\}  }{{{\left\vert {\nabla
{u_{n}}}\right\vert }^{p}}{\psi_{1,\delta}}}=\omega(n,m,\delta),\qquad
\int\limits_{\left\{  -m-k<{U_{n}}\leq-m\right\}  }{{{\left\vert {\nabla
{U_{n}}}\right\vert }^{p}}{\psi_{1,\delta}}}=\omega(n,m,\delta).\label{13054}%
\end{equation}

\end{lemma}

\begin{proof}
(i) Proof of (\ref{13051}), (\ref{13052}). Set for any $r\in\mathbb{R}$ and
any $m,\ell\geq1$%
\[
{S_{m,\ell}}(r)=\int_{0}^{r}{\left(  {\frac{{-m+\tau}}{m}{\chi_{\lbrack
m,2m]}}(\tau)+{\chi_{(2m,2m+\ell]}}(\tau)+\frac{{4m+2h-\tau}}{{2m+\ell}}%
{\chi_{(2m+\ell,4m+2h]}}(\tau)}\right)  d\tau,}%
\]%
\[
{S_{m}}(r){=}\int_{0}^{r}{\left(  {\frac{{-m+\tau}}{m}{\chi_{\lbrack m,2m]}%
}(\tau)+{\chi_{(2m,\infty)}}(\tau)}\right)  d\tau}.
\]
Note that ${S}_{m,\ell}^{\prime\prime}{=\chi}_{\left[  m,2m\right]  }/m-{\chi
}_{\left[  2m+\ell,2(2m+\ell)\right]  }/(2m+\ell).$ We choose $(\xi
,J,S)=(\psi_{2,\delta},T_{1},S_{m,\ell})$ as test functions in (\ref{partu})
for $u_{n},$ and observe that, from (\ref{muni}),
\begin{equation}
\widehat{\mu_{n,0}}=\mu_{n,0}-(h_{n})_{t}=\widehat{\lambda_{n,0}}+\rho
_{n,0}-\eta_{n,0}=f_{n}-\operatorname{div}g_{n}+\rho_{n,0}-\eta_{n,0}%
.\label{xxx}%
\end{equation}
Thus we can write $%
{\textstyle\sum_{i=1}^{6}}
A_{i}\leq%
{\textstyle\sum_{i=7}^{12}}
A_{i},$ where
\begin{align*}
A_{1} &  =-\int_{\Omega}{{\psi_{2,\delta}}(0){T_{1}}({S_{m,\ell}}({u_{0,n}%
})){S_{m,\ell}}({u_{0,n}})dx,\quad}A_{2}=-\int_{Q}{{{\left(  {{\psi_{2,\delta
}}}\right)  }_{t}}\overline{T_{1}}({S_{m,\ell}}({U_{n}})),}\\
A_{3} &  =\int_{Q}S_{m,\ell}^{\prime}{({U_{n}}){T_{1}}({S_{m,\ell}}({U_{n}%
}))A(x,t,\nabla u_{n})\nabla{\psi_{2,\delta}},\quad}A_{4}=\int_{Q}(S_{m,\ell
}^{\prime}{{{{({U_{n}}))}}^{2}\psi_{2,\delta}}{T}}_{1}^{\prime}{({S_{m,\ell}%
}({U_{n}}))A(x,t,\nabla u_{n})\nabla{U_{n},}}\\
A_{5} &  =\frac{1}{m}\int\limits_{\left\{  m\leq{U_{n}}\leq2m\right\}  }%
{{\psi_{2,\delta}}{T_{1}}({S_{m,\ell}}({U_{n}}))A(x,t,\nabla u_{n}%
)\nabla{U_{n}},}%
\end{align*}%
\begin{align*}
A_{6} &  =-\frac{1}{{2m+\ell}}\int\limits_{\left\{  2m+\ell\leq{U_{n}%
}<2(2m+\ell)\right\}  }{\psi_{2,\delta}A(x,t,\nabla u_{n})\nabla{U_{n},}}\\
A_{7} &  =\int_{Q}S_{m,\ell}^{\prime}{({U_{n}}){T_{1}}({S_{m,\ell}}({U_{n}%
})){\psi_{2,\delta}}{f_{n},\quad\quad}}A_{8}=\int_{Q}S_{m,\ell}^{\prime
}{({U_{n}}){T_{1}}({S_{m,\ell}}({U_{n}})){g_{n}.}\nabla{\psi_{2,\delta},}}\\
A_{9} &  =\int_{Q}{{{\left(  S_{m,\ell}^{\prime}{({U_{n}})}\right)  }^{2}%
}{T_{1}^{^{\prime}}}({S_{m,\ell}}({U_{n}})){\psi_{2,\delta}g_{n}.}\nabla
{U_{n},\quad\quad}}A_{10}=\frac{1}{m}\int\limits_{m\leq{U_{n}}\leq2m}{{T_{1}%
}({S_{m,\ell}}({U_{n}})){\psi_{2,\delta}g_{n}.}\nabla{U_{n},}}\\
A_{11} &  =-\frac{1}{{2m+\ell}}\int\limits_{\left\{  2m+\ell\leq{U_{n}%
}<2(2m+\ell)\right\}  }{{\psi_{2,\delta}g_{n}.}\nabla{U_{n}},\quad}A_{12}%
=\int_{Q}S_{m,\ell}^{\prime}{({U_{n}}){T_{1}}({S_{m,\ell}}({U_{n}}%
)){\psi_{2,\delta}}d\left(  {{\rho_{n,0}}-{\eta_{n,0}}}\right)  .}%
\end{align*}
Since $||S_{m,\ell}(u_{0,n})||_{1,\Omega}\leq\int\limits_{\left\{  m\leq
u_{0,n}\right\}  }u_{0,n}dx$, we find $A_{1}=\omega(\ell,n,m)$. Otherwise
\[
|A_{2}|\leq{\left\Vert {{\psi_{2,\delta}}}\right\Vert _{{W^{1,\infty}}({Q})}%
}\int\limits_{\left\{  m\leq U_{n}\right\}  }{{U_{n}}},\qquad|A_{3}%
|\leq{\left\Vert {{\psi_{2,\delta}}}\right\Vert _{{W^{1,\infty}}({Q})}}%
\int\limits_{\left\{  m\leq U_{n}\right\}  }\left(  |a|+\Lambda_{2}{\left\vert
{\nabla{u_{n}}}\right\vert }^{p-1}\right)  ,
\]
which imply $A_{2}=\omega(\ell,n,m)$ and $A_{3}=\omega(\ell,n,m).$ Using
(\ref{renor2}) for $u_{n}$, we have
\[
A_{6}=-\int_{Q}{{\psi_{2,\delta}}d{{\left(  {{\rho_{n,s}}-{\eta_{n,s}}%
}\right)  }^{+}}}+\omega(\ell)=\omega(\ell,n,m,\delta).
\]
Hence $A_{6}=\omega(\ell,n,m,\delta),$ since ${{{\left(  {{\rho_{n,s}}%
-{\eta_{n,s}}}\right)  }^{+}}}$ converges to $\mu_{s}^{+}$ as $n\rightarrow
\infty$ in the narrow topology, and $\int_{Q}{\psi_{2,\delta}}d\mu_{s}^{+}%
\leq\delta.$ We also obtain $A_{11}=\omega(\ell)$ from (\ref{lim3}).

\noindent Now $\left\{  S_{m,\ell}^{\prime}(U_{n})T_{1}(S_{m,\ell}%
(U_{n}))\right\}  _{\ell}$ converges to $S_{m}^{\prime}(U_{n})T_{1}%
(S_{m}(U_{n}))$, $\left\{  S_{m}^{\prime}(U_{n})T_{1}(S_{m}(U_{n}))\right\}
_{n}$ converges to $S_{m}^{\prime}(U)$ $T_{1}(S_{m}(U))$, $\left\{
S_{m}^{\prime}(U)T_{1}(S_{m}(U))\right\}  _{m}$ converges to $0$,
weak-$^{\ast}$ in $L^{\infty}(Q)$  and $\left\{  f_{n}\right\}  $ converges to
$f$ weakly in $L^{1}(Q)$, $\left\{  g_{n}\right\}  $ converges to $g$ strongly
in $(L^{p^{\prime}}(Q))^{N}$. From Remark \ref{05041}, we obtain%
\begin{align*}
A_{7} &  =\int_{Q}S_{m}^{\prime}{({U_{n}}){T_{1}}({S_{m}}({U_{n}}%
)){\psi_{2,\delta}}{f_{n}}}+\omega(\ell)=\int_{Q}S_{m}^{\prime}{(U){T_{1}%
}({S_{m}}(U)){\psi_{2,\delta}}f}+\omega(\ell,n)=\omega(\ell,n,m),\\
A_{8} &  =\int_{Q}S_{m}^{\prime}{({U_{n}}){T_{1}}({S_{m}}({U_{n}})){g_{n}%
.}\nabla{\psi_{2,\delta}}}+\omega(\ell)=\int_{Q}S_{m}^{\prime}{(U){T_{1}%
}({S_{m}}(U))g\nabla{\psi_{2,\delta}}}+\omega(\ell,n)=\omega(\ell,n,m).
\end{align*}
\newline Otherwise, $A_{12}\leq\int_{Q}{{\psi_{2,\delta}}d{\rho_{n}}}$, and
$\left\{  \int_{Q}{{\psi_{2,\delta}}d{\rho_{n}}}\right\}  $ converges to
$\int_{Q}{\psi_{2,\delta}}d\mu_{s}^{+},$ thus $A_{12}\leq\omega(\ell
,n,m,\delta)$.

\noindent Using Holder inequality and the condition (\ref{condi1}), we have
\[
g_{n}.\nabla U_{n}-A(x,t,\nabla u_{n})\nabla{U_{n}}\leq c_{1}\left(
|g_{n}|^{p^{\prime}}+|\nabla h_{n}|^{p}+|a|^{p^{\prime}}\right)
\]
with $c_{1}=c_{1}(p,\Lambda_{1},\Lambda_{2}),$ which implies
\[
A_{9}-A_{4}\leq c_{1}\int_{Q}{{{\left(  S_{m,\ell}^{\prime}{({U_{n}})}\right)
}^{2}T}}_{1}^{\prime}{({S_{m,\ell}}({U_{n}})){\psi_{2,\delta}}\left(
{|{g_{n}}{|^{p^{\prime}}}+|{h_{n}}{|^{p}}}+|a|^{p^{\prime}}\right)  =\;}%
\omega(\ell,n,m).
\]
Similarly we also show that $A_{10}-A_{5}/2\leq\omega(\ell,n,m)$. Combining
the estimates, we get $A_{5}/2\leq\omega(\ell,n,m,\delta)$. Using Holder
inequality we have
\[
A(x,t,\nabla u_{n})\nabla U_{n}\geq\frac{\Lambda_{1}}{2}|\nabla u_{n}%
|^{p}-c_{2}(|a|^{p^{\prime}}+|\nabla h_{n}|^{p}).
\]
with $c_{2}=c_{2}(p,\Lambda_{1},\Lambda_{2}),$ which implies
\[
\frac{1}{m}\int\limits_{\left\{  m\leq{U_{n}}<2m\right\}  }{{{\left\vert
{\nabla{u_{n}}}\right\vert }^{p}}{\psi_{2,\delta}}{T_{1}}({S_{m,\ell}}({U_{n}%
}))=\;}\omega(\ell,n,m,\delta).
\]
Note that for all $m$ $>4$, $S_{m,\ell}(r)\geq1$ for any $r\in\lbrack\frac
{3}{2}m,2m];$ hence $T_{1}(S_{m,\ell}(r))=1.$ So,
\[
\frac{1}{m}\int\limits_{\left\{  \frac{3}{2}m\leq{U_{n}}<2m\right\}
}{{{\left\vert {\nabla{u_{n}}}\right\vert }^{p}}{\psi_{2,\delta}}}=\omega
(\ell,n,m,\delta).
\]
Since ${\left\vert {\nabla{U_{n}}}\right\vert ^{p}}\leq{2^{p-1}}{\left\vert
{\nabla{u_{n}}}\right\vert ^{p}}+{2^{p-1}}{\left\vert {\nabla{h_{n}}%
}\right\vert ^{p}}$, there also holds
\[
\frac{1}{m}\int\limits_{\left\{  \frac{3}{2}m\leq{U_{n}}<2m\right\}
}{{{\left\vert {\nabla{U_{n}}}\right\vert }^{p}}{\psi_{2,\delta}}}=\omega
(\ell,n,m,\delta).
\]
We deduce (\ref{13051}) by summing on each set $\left\{  (\frac{4}{3}%
)^{i}m\leq{U_{n}}\leq(\frac{4}{3})^{i+1}m\right\}  $ for $i=0,1,2.$ Similarly,
we can choose $(\xi,\psi,S)=(\psi_{1,\delta},T_{1},\tilde{S}_{m,\ell})$ as
test functions in (\ref{partu}) for $u_{n},$ where $\tilde{S}_{m,\ell}(r)=$
${S_{m,\ell}}(-r),$ and we obtain (\ref{13052}).\medskip\ 

(ii) Proof of (\ref{13053}), (\ref{13054}). We set, for any $k,m,\ell\geq1,$%
\[
{S_{k,m,\ell}}(r)=\int_{0}^{r}{\left(  {{T_{k}}(\tau-{T_{m}}(\tau
)){\chi_{\lbrack m,k+m+\ell]}}+k\frac{{2(k+\ell+m)-\tau}}{{k+m+\ell}}%
{\chi_{(k+m+\ell,2(k+m+\ell)]}}}\right)  d\tau}%
\]%
\[
{S_{k,m}}(r)=\int_{0}^{r}{{T_{k}}(\tau-{T_{m}}(\tau)){\chi_{\lbrack m,\infty
)}}d\tau.}%
\]
We choose $(\xi,\psi,S)=(\psi_{2,\delta},T_{1},S_{k,m,\ell})$ as test
functions in (\ref{partu}) for $u_{n}$. In the same way we also obtain
\[
\int\limits_{\left\{  m\leq{U_{n}}<m+k\right\}  }{{{\left\vert {\nabla{u_{n}}%
}\right\vert }^{p}}{\psi_{2,\delta}}{T_{1}}({S_{k,m,\ell}}({U_{n}}))}%
=\omega(\ell,n,m,\delta).
\]
Note that $T_{1}(S_{k,m,\ell}(r))$ $=1$ for any $r$ $\geq m+1$, thus
$\int\limits_{\left\{  m+1\leq{U_{n}}<m+k\right\}  }{{{\left\vert
{\nabla{u_{n}}}\right\vert }^{p}}{\psi_{2,\delta}}}=\omega(n,m,\delta),$ which
implies (\ref{13053}) by changing $m$ into $m-1$. Similarly, we obtain
(\ref{13054}).\medskip
\end{proof}

Next we look at the behaviour near $E.$

\begin{lemma}
\label{near} Estimate (\ref{12059}) holds.
\end{lemma}

\begin{proof}
There holds%
\[
I_{1}=\int_{Q}{\Phi_{\delta_{1},\delta_{2}}A(x,t,\nabla u_{n}).\nabla{T_{k}%
}({U_{n}})-}\int\limits_{\left\{  |U_{n}|\leq k\right\}  }{\Phi_{\delta
_{1},\delta_{2}}A(x,t,\nabla u_{n}).\nabla{\langle T_{k}(U)\rangle}_{\nu}{.}}%
\]
From Proposition \ref{mun}, (iv), $\left\{  A(x,t,\nabla\left(  T_{k}%
(U_{n})+h_{n}\right)  ).\nabla\langle T_{k}(U)\rangle_{\nu}\right\}  $
converges weakly in $L^{1}(Q)$ to $F_{k}\nabla\langle T_{k}(U)\rangle_{\nu}.$
And $\left\{  \chi_{\left\{  |U_{n}|\leq k\right\}  }\right\}  $ converges to
$\chi_{|U|\leq k},$ $a.e.$ in $Q$ , and $\Phi_{\delta_{1},\delta_{2}}$
converges to $0$ $a.e.$ in $Q$ as $\delta_{1}\rightarrow0,$ and $\Phi
_{\delta_{1},\delta_{2}}$ takes its values in $\left[  0,1\right]  $. From
Remark \ref{05041}, we have
\begin{align*}
&  \int\limits_{\left\{  |U_{n}|\leq k\right\}  }{{\Phi_{{\delta_{1}}%
,{\delta_{2}}}}A(x,t,\nabla u_{n}).\nabla{\langle T_{k}(U)\rangle}_{\nu}}%
=\int_{Q}{\chi_{\left\{  |U_{n}|\leq k\right\}  }{\Phi_{{\delta_{1}}%
,{\delta_{2}}}}A(x,t,\nabla\left(  T_{k}(U_{n})+h_{n}\right)  ).\nabla\langle
T_{k}(U)\rangle}_{\nu}\\
&  =\int_{Q}{\chi_{|U|\leq k}{\Phi_{{\delta_{1}},{\delta_{2}}}}F_{k}%
.\nabla{\langle T_{k}(U)\rangle}_{\nu}}+\omega(n)=\omega(n,\nu,{\delta_{1}}).
\end{align*}
Therefore, if we prove that
\begin{equation}
\int_{Q}{\Phi_{\delta_{1},\delta_{2}}A(x,t,\nabla u_{n}).\nabla{T_{k}}({U_{n}%
})}\leq\omega(n,\delta_{1},\delta_{2}), \label{120511}%
\end{equation}
then we deduce (\ref{12059}). As noticed in \cite{DMOP,Pe08}, it is precisely
for this estimate that we need the double cut ${\psi_{{\delta_{1}}}^{+}%
\psi_{{\delta_{2}}}^{+}.}$ To do this, we set, for any $m>k>0,$ and any
$r\in\mathbb{R},$%
\[
{\hat{S}_{k,m}}(r)=\int_{0}^{r}{\left(  {k-{T_{k}}(\tau)}\right)  H{_{m}}%
(\tau)d\tau,}%
\]
where $H{_{m}}$ is defined at (\ref{Hm}). Hence supp ${\hat{S}_{k,m}\subset
}\left[  -2m,k\right]  ;$ and ${\hat{S}}_{k,m}^{\prime\prime}{=-\chi}_{\left[
-k,k\right]  }+\frac{2k}{m}{\chi}_{\left[  -2m,-m\right]  }.$ We choose
$(\varphi,S)=({\psi_{{\delta_{1}}}^{+}\psi_{{\delta_{2}}}^{+}},{\hat{S}_{k,m}%
})$ as test functions in (\ref{renor}). From (\ref{xxx}), we can write
\[
{A}_{1}+{A}_{2}-{A}_{3}+{A}_{4}+{A}_{5}+{A}_{6}=0,
\]
where
\begin{align*}
&  {A}_{1}=-\int_{Q}({{{{\psi_{{\delta_{1}}}^{+}\psi_{{\delta_{2}}}^{+}}})}%
}_{t}{\hat{S}_{k,m}({U_{n}}),\quad A}_{2}=\int_{Q}{(k-{T_{k}}({U_{n}}))H{_{m}%
}({U_{n}})A(x,t,\nabla u_{n}).\nabla({\psi_{{\delta_{1}}}^{+}\psi_{{\delta
_{2}}}^{+})},}\\
{A}_{3}  &  =\int_{Q}{\psi_{{\delta_{1}}}^{+}\psi_{{\delta_{2}}}%
^{+}A(x,t,\nabla u_{n}).\nabla{T_{k}}({U_{n}}),\quad A}_{4}=\frac{{2k}}{m}%
\int\limits_{\left\{  -2m<{U_{n}}\leq-m\right\}  }{\psi_{{\delta_{1}}}^{+}%
\psi_{{\delta_{2}}}^{+}A(x,t,\nabla u_{n}).\nabla{U_{n}},}\\
{A}_{5}  &  =-\int_{Q}{(k-{T_{k}}({U_{n}}))H{_{m}}({U_{n}})\psi_{{\delta_{1}}%
}^{+}\psi_{{\delta_{2}}}^{+}}d\widehat{\lambda_{n,0}},\quad{A}_{6}=\int%
_{Q}{(k-{T_{k}}({U_{n}})){H_{m}}({U_{n}})\psi_{{\delta_{1}}}^{+}\psi
_{{\delta_{2}}}^{+}d\left(  {{\eta_{n,0}-\rho_{n,0}}}\right)  .}%
\end{align*}

We first estimate ${A}_{3}.$ As in \cite[p.585]{Pe08}, since $\left\{
{{\hat{S}_{k,m}}({U_{n}})}\right\}  $ converges to {{$\hat{S}_{k,m}$}}${(}%
${{$U$}}${)}$ weakly in $X,$ and {{$\hat{S}_{k,m}$}}${(}${{$U$}}${)\in
L}^{\infty}(Q),$ using (\ref{41}), we find
\[
{A}_{1}=-\int_{Q}({{{{\psi_{{\delta_{1}}}^{+})}}}}_{t}{{{{\psi_{{\delta_{2}}%
}^{+}}}\hat{S}_{k,m}}({U})-}\int_{Q}{{{{\psi_{{\delta_{1}}}^{+}}}}}%
({{{{\psi_{{\delta_{2}}}^{+})}}}}_{t}{{\hat{S}_{k,m}}({U})+\omega(n)=\omega
}(n,\delta_{1}).
\]

Next consider ${A}_{2}.$ Notice that {{$U$}}${{_{n}=}}T_{2m}(U_{n})$ on supp
$({H{_{m}}(}${{$U$}}${{_{n}})})$. From Proposition \ref{mun}, (iv), the
sequence $\left\{  A(x,t,\nabla\left(  T_{2m}(U_{n})+h_{n}\right)
).\nabla(\psi_{\delta_{1}}^{+}\psi_{\delta_{2}}^{+})\right\}  $ converges to
$F_{2m}.\nabla(\psi_{\delta_{1}}^{+}\psi_{\delta_{2}}^{+})$ weakly in
$L^{1}(Q)$. From Remark \ref{05041} and the convergence of $\psi_{\delta_{1}%
}^{+}\psi_{\delta_{2}}^{+}$ in $X$ to $0$ as $\delta_{1}$ tends to $0$, we
find
\[
{A}_{2}=\int_{Q}{(k-{T_{k}}(U)){H_{m}}(U)F_{2m}.\nabla({\psi_{{\delta_{1}}%
}^{+}\psi_{{\delta_{2}}}^{+})}}+\omega(n)=\omega(n,{\delta_{1}}).
\]

Then consider ${A}_{4}.$ Then for some $c_{1}=c_{1}(p,\Lambda_{2}),$
\[
\left\vert {A}_{4}\right\vert \leq c_{1}\frac{{2k}}{m}\int\limits_{\left\{
-2m<{U_{n}}\leq-m\right\}  }\left(  |\nabla{u_{n}}|^{p}+|\nabla{U_{n}}%
|^{p}+|a|^{p^{\prime}}\right)  \psi_{{\delta_{1}}}^{+}\psi_{{\delta_{2}}}%
^{+}.
\]
Since ${\psi_{{\delta_{1}}}^{+}}$ takes its values in $\left[  0,1\right]  ,$
from Lemma \ref{april261}, we get in particular ${A}_{4}=\omega(n,\delta
_{1},m,\delta_{2})$.

Now we estimate $A_{5}.$ The sequence $\left\{  (k-T_{k}(U_{n})){H{_{m}%
}({U_{n}})\psi_{{\delta_{1}}}^{+}\psi_{{\delta_{2}}}^{+}}\right\}  $ converges
to $(k-T_{k}(U)){H{_{m}}(}${{$U$}}${)\psi_{{\delta_{1}}}^{+}\psi_{{\delta_{2}%
}}^{+},}$ weakly in $X,$ and $\left\{  (k-T_{k}(U_{n}))H_{m}(U_{n})\right\}  $
converges to $(k-T_{k}(U))H_{m}(U),$ weak-$^{\ast}$ in $L^{\infty}(Q)$  and
$a.e.$ in $Q.$ Otherwise $\left\{  f_{n}\right\}  $ converges to $f$ weakly in
$L^{1}\left(  Q\right)  $ and $\left\{  g_{n}\right\}  $ converges to $g$
strongly in $(L^{p^{\prime}}\left(  Q\right)  )^{N}.$ From Remark \ref{05041}
and the convergence of ${\psi_{{\delta_{1}}}^{+}\psi_{{\delta_{2}}}^{+}}$ to
$0$ in $X$ and $a.e.$ in $Q$ as $\delta_{1}\rightarrow0$, we deduce that
\[
{A}_{5}=-\int_{Q}{(k-{T_{k}}({U_{n}}))H{_{m}}({U})\psi_{{\delta_{1}}}^{+}%
\psi_{{\delta_{2}}}^{+}}d\widehat{\nu_{0}}+\omega(n)=\omega(n,\delta_{1}),
\]
where $\widehat{\nu_{0}}=f-\operatorname{div}g.$

Finally ${A}_{6}\leq2k\int_{Q}{\psi_{{\delta_{1}}}^{+}\psi_{{\delta_{2}}}%
^{+}d{\eta_{n}}}$; using (\ref{12054}) we also find ${A}_{6}$ $\leq
\omega(n,\delta_{1},m,\delta_{2}).$ By addition, since ${A}_{3}$ does not
depend on $m,$ we obtain
\[
{A}_{3}=\int_{Q}{\psi_{{\delta_{1}}}^{+}\psi_{{\delta_{2}}}^{+}A(x,t,\nabla
u_{n})\nabla{T_{k}}({U_{n}})}\leq\omega(n,{\delta_{1}},{\delta_{2}}).
\]
Arguying as before with $({\psi_{{\delta_{1}}}^{-}\psi_{{\delta_{2}}}^{-}%
},{\check{S}_{k,m}})$ as test function in (\ref{renor}), where ${\check
{S}_{k,m}(r)=-}${$\hat{S}_{k,m}$}$(-r),$ we get in the same way
\[
\int_{Q}{\psi_{{\delta_{1}}}^{-}\psi_{{\delta_{2}}}^{-}A(x,t,\nabla
u_{n})\nabla{T_{k}}({U_{n}})}\leq\omega(n,{\delta_{1}},{\delta_{2}}).
\]
Then, (\ref{120511}) holds.\medskip
\end{proof}

Next we look at the behaviour far from $E.$

\begin{lemma}
\label{far}. Estimate (\ref{120510}) holds.
\end{lemma}

\begin{proof}
Here we estimate $I_{2};$ we can write
\[
I_{2}=\int\limits_{\left\{  |U_{n}|\leq k\right\}  }{(1-\Phi_{\delta
_{1},\delta_{2}})A(x,t,\nabla u_{n})\nabla\left(  {{T_{k}}({U_{n}})-}\langle
T_{k}(U)\rangle_{\nu}\right)  .}%
\]

\noindent Following the ideas of \cite{Por99}, used also in \cite{Pe08}, we
define, for any $r\in\mathbb{R}$ and $\ell>2k>0$,
\[
{R_{n,\nu,\ell}}={T_{\ell+k}}\left(  {{U_{n}}-}\langle T_{k}(U)\rangle_{\nu
}\right)  -{T_{\ell-k}}\left(  {{U_{n}}-{T_{k}}\left(  {{U_{n}}}\right)
}\right)  .
\]
Recall that $\left\Vert \langle T_{k}(U)\rangle_{\nu}\right\Vert _{\infty
,Q}\leq k,$ and observe that
\begin{equation}
{R_{n,\nu,\ell}}=2k\;\mathrm{sign}({U_{n}})\quad\text{in}\;\left\{
{\left\vert {{U_{n}}}\right\vert }\geq{\ell+2k}\right\}  ,\quad\text{
}|R_{n,\nu,\ell}|\leq4k,\quad R_{n,\nu,\ell}=\omega(n,\nu,\ell)\text{
}a.e.\text{ in }Q,\label{13057}%
\end{equation}%
\begin{equation}
\lim_{n\rightarrow\infty}R_{n,\nu,\ell}={T_{\ell+k}}\left(  {{U}-\langle
T_{k}(U)\rangle}_{\nu}\right)  -{T_{\ell-k}}\left(  {{U}-{T_{k}}\left(  {{U}%
}\right)  }\right)  ,\qquad a.e.\;\text{in}\;{Q},\text{ and weakly in
}X.\label{13058}%
\end{equation}
Next consider $\xi_{1,n_{1}}\in C_{c}^{\infty}([0,T)),\xi_{2,n_{2}}\in
C_{c}^{\infty}((0,T])$ with values in $[0,1],$ such that $(\xi_{1,n_{1}}%
)_{t}\leq0$ and $(\xi_{2,n_{2}})_{t}$ $\geq0$; and $\left\{  \xi_{1,n_{1}%
}(t)\right\}  $ (resp. $\left\{  \xi_{1,n_{2}}(t)\right\}  )$ converges to
$1,\,$for any $t\in\lbrack0,T)$ (resp. $t\in(0,T]$ ); and moreover, for any
$a\in C([0,T];L^{1}(\Omega))$, $\left\{  \int_{Q}a{{{\left(  \xi_{1,n_{1}%
}\right)  }_{t}}}\right\}  $ and $\int_{Q}a{{{\left(  \xi_{2,n_{2}}\right)
}_{t}}}$ converge respectively to $-\int_{\Omega}{a(.,T)dx}$ and $\int%
_{\Omega}{a(.,0)dx.}$ We set
\[
\varphi={\varphi_{n,n_{1},n_{2},{l_{1}},{l_{2},\ell}}}=\xi_{1,n_{1}}%
(1-{\Phi_{\delta_{1},\delta_{2}}}){\left[  {{T_{\ell+k}}\left(  {{U_{n}}%
-}\langle T_{k}(U)\rangle_{\nu}\right)  }\right]  _{{l_{1}}}}-\xi_{2,n_{2}%
}(1-{\Phi_{\delta_{1},\delta_{2}}}){\left[  {{T_{\ell-k}}\left(
{U_{n}-{{{T_{k}}(U_{n})}}}\right)  }\right]  _{{-l_{2}}}.}%
\]
We observe that
\begin{equation}
{\varphi-(1-{\Phi_{\delta_{1},\delta_{2}}}){R_{n,\nu,\ell}}}=\omega
(l_{1},l_{2},n_{1},n_{2})\;\quad\text{ in norm in }X\text{ and }%
a.e.\;\text{in}\;{Q}.\label{13056}%
\end{equation}
We can choose $(\varphi,S)=({\varphi_{n,n_{1},n_{2},{l_{1}},{l_{2},\ell}}%
},\overline{H_{m}})$ as test functions in (\ref{renor4}) for $u_{n}$, where
$\overline{H_{m}}$ is defined at (\ref{Hm}), with $m>\ell+2k.$ We obtain
\[
A_{1}+A_{2}+A_{3}+A_{4}+A_{5}=A_{6}+A_{7},
\]
with
\begin{align*}
A_{1} &  =\int_{\Omega}{\varphi(T){\overline{H_{m}}}({U_{n}(T)})dx,\quad\quad
}A_{2}=-\int_{\Omega}{\varphi(0){\overline{H_{m}}}({u_{0,n}})dx,\quad\quad
}A_{3}=-\int_{Q}{\varphi{_{t}\overline{H_{m}}}({U_{n}}),}\\
A_{4} &  =\int_{Q}H_{m}{({U_{n}})A(x,t,\nabla u_{n}).\nabla\varphi}%
,{\quad\quad}A_{5}=\int_{Q}{{\varphi}}H_{m}^{\prime}{({U_{n}})A(x,t,\nabla
u_{n}).\nabla U_{n}{,}}\\
A_{6} &  =\int_{Q}H_{m}{({U_{n}}){\varphi}d}\widehat{\lambda_{n,0}}%
{{,\quad\quad}}A_{7}=\int_{Q}H_{m}{({U_{n}}){\varphi}d\left(  {{\rho_{n,0}%
}-{\eta_{n,0}}}\right)  .}%
\end{align*}
\newline\textbf{Estimate} \textbf{of} $A_{4}$. This term allows to study
$I_{2}.$ Indeed, $\left\{  H_{m}(U_{n})\right\}  $ converges to $1,$ $a.e.$ in
$Q$; From (\ref{13056}), (\ref{13057}) (\ref{13058}), we have
\begin{align*}
A_{4} &  =\int_{Q}{(1-{\Phi_{\delta_{1},\delta_{2}}})A(x,t,\nabla
u_{n}).\nabla{R_{n,\nu,\ell}}}-\int_{Q}{{R_{n,\nu,\ell}}A(x,t,\nabla
u_{n}).\nabla{\Phi_{\delta_{1},\delta_{2}}+}}\omega(l_{1},l_{2},n_{1}%
,n_{2},m)\\
&  =\int_{Q}{(1-{\Phi_{\delta_{1},\delta_{2}}})A(x,t,\nabla u_{n}%
).\nabla{R_{n,\nu,\ell}+}}\omega(l_{1},l_{2},n_{1},n_{2},m,n,\nu,\ell)\\
&  =I_{2}+\int\limits_{\left\{  \left\vert {{U_{n}}}\right\vert >k\right\}
}{(1-{\Phi_{\delta_{1},\delta_{2}}})A(x,t,\nabla u_{n}).\nabla{R_{n,\nu,\ell
}+}}\omega(l_{1},l_{2},n_{1},n_{2},m,n,\nu,\ell)\\
&  =I_{2}+B_{1}+B_{2}+\omega(l_{1},l_{2},n_{1},n_{2},m,n,\nu,\ell),
\end{align*}
where
\begin{align*}
B_{1} &  =\int\limits_{\left\{  \left\vert {{U_{n}}}\right\vert >k\right\}
}{(1-{\Phi_{\delta,\eta}})({{\chi_{\left\vert {{U_{n}}-\langle T_{k}%
(U)\rangle}_{\nu}\right\vert \leq\ell+k}}-{\chi_{\left\vert \left\vert
{{U_{n}}}\right\vert {-k}\right\vert \leq\ell-k})}}A(x,t,\nabla u_{n}).\nabla
U_{n},}\\
B_{2} &  =-\int\limits_{\left\{  \left\vert {{U_{n}}}\right\vert >k\right\}
}(1-{\Phi_{\delta_{1},\delta_{2}}}){\chi_{\left\vert {{U_{n}}-\langle
T_{k}(U)\rangle}_{\nu}\right\vert \leq\ell+k}}A(x,t,\nabla u_{n}%
).\nabla\langle{{{{{T_{k}}(U)\rangle}}_{\nu}.}}%
\end{align*}
Now $\left\{  A(x,t,\nabla\left(  {{T_{\ell+2k}}({U_{n}})+{h_{n}}}\right)
).\nabla{\langle T_{k}(U)\rangle}_{\nu}\right\}  $ converges to $F_{\ell
+2k}\nabla{\langle T_{k}(}${$U$}${)\rangle}_{\nu},$ weakly in $L^{1}(Q).$
Otherwise $\left\{  \chi_{|U_{n}|>k}{\chi_{\left\vert {{U_{n}}-\langle
T_{k}(U)\rangle}_{\nu}\right\vert \leq\ell+k}}\right\}  $ converges to
$\chi_{|U|>k}{\chi_{\left\vert {{U}-\langle T_{k}(U)\rangle}_{\nu}\right\vert
\leq\ell+k},}$ $a.e.$ in $Q$. And $\left\{  \langle T_{k}(U)\rangle_{\nu
}\right\}  $ converges to $T_{k}(U)$ strongly in $X$. From Remark \ref{05041}
we get%
\begin{align*}
B_{2} &  =-\int_{Q}{(1-{\Phi_{{\delta_{1}},{\delta_{2}}}})\;{\chi_{|U|>k}%
\;}{\chi_{\left\vert {U-\langle T_{k}(U)\rangle}_{\nu}\right\vert \leq\ell+k}%
}F_{\ell+2k}.\nabla{\langle T_{k}(U)\rangle}_{\nu}}+\omega(n)\\
&  =-\int_{Q}{(1-{\Phi_{{\delta_{1}},{\delta_{2}}}}){{{\;}}}{\chi_{|U|>k}%
\;}{\chi_{\left\vert {U-{{{{T_{k}}(U)}}}}\right\vert \leq\ell+k}}F_{\ell
+2k}.\nabla{{{{T_{k}}(U)}}}}+\omega(n,\nu)=\omega(n,\nu),
\end{align*}
since ${\nabla{{{{T_{k}}(}}}}${{{{$U$}}}}${{{{)\;}}}{\chi_{|U|>k}=0.}}$
Besides, we see that, for some $c_{1}=c_{1}(p,\Lambda_{2}),$%
\[
\left\vert B_{1}\right\vert \leq c_{1}\int\limits_{\left\{  \ell
-2k\leq\left\vert {{U_{n}}}\right\vert <\ell+2k\right\}  }{(1-{\Phi
_{\delta_{1},\delta_{2}}})(|\nabla u_{n}|^{p}+|\nabla U_{n}|^{p}+|a|^{p}%
}^{\prime}).
\]
Using (\ref{12056}) and (\ref{12057}) and applying (\ref{13053}) and
(\ref{13054}) to ${1-{\Phi_{\delta_{1},\delta_{2}}}}$, we obtain, for $k>0,$
\begin{equation}
\int\limits_{\left\{  m\leq|{U_{n}}|<m+4k\right\}  }({{{\left\vert
{\nabla{u_{n}}}\right\vert }^{p}+{\left\vert {\nabla{U_{n}}}\right\vert }%
^{p})}(1-{\Phi_{\delta_{1},\delta_{2}}})}=\omega(n,m,\delta_{1},\delta
_{2}).\label{04043}%
\end{equation}
Thus, $B_{1}=\omega(n,\nu,\ell,\delta_{1},\delta_{2}),$ hence $B_{1}%
+B_{2}=\omega(n,\nu,\ell,\delta_{1},\delta_{2}).$ Then
\begin{equation}
A_{4}=I_{2}+\omega(l_{1},l_{2},n_{1},n_{2},m,n,\nu,\ell,\delta_{1},\delta
_{2}).\label{a4}%
\end{equation}
\textbf{Estimate of} $A_{5}$. For $m>\ell+2k$, since $|\varphi|\leq2\ell,$ and
(\ref{13056}) holds, we get, from the dominated convergence Theorem,
\begin{align*}
A_{5} &  =\int_{Q}(1-{\Phi_{\delta_{1},\delta_{2}}})R_{n,\nu,\ell}%
H_{m}^{\prime}(U_{n})A(x,t,\nabla u_{n}).\nabla U_{n}+\omega(l_{1},l_{2}%
,n_{1},n_{2})\\
&  =-\frac{{2k}}{m}\int\limits_{\left\{  m\leq\left\vert {{U_{n}}}\right\vert
<2m\right\}  }{(1-{\Phi_{\delta_{1},\delta_{2}}})A(x,t,\nabla u_{n}).\nabla
U_{n}+}\omega(l_{1},l_{2},n_{1},n_{2});
\end{align*}
here, the final equality followed from the relation, since $m>\ell+2k,$
\begin{equation}
R_{n,\nu,\ell}H_{m}^{\prime}(U_{n})=-\frac{2k}{m}\chi_{m\leq|U_{n}|\leq
2m},\quad a.e.\text{ in }Q.\label{relt}%
\end{equation}
Next we go to the limit in $m,$ by using (\ref{renor2}), (\ref{renor3}) for
$u_{n}$, with ${\phi=(1-{\Phi_{\delta_{1},\delta_{2}}})}$. There holds
\[
A_{5}=-2k\int_{Q}{(1-{\Phi_{\delta_{1},\delta_{2}}})d\left(  (\rho_{n,s}%
-\eta_{n,s})^{+}+(\rho_{n,s}-\eta_{n,s})^{-}\right)  +}\omega(l_{1}%
,l_{2},n_{1},n_{2},m).
\]
Then, from (\ref{12056}) and (\ref{12057}), we get $A_{5}=\omega(l_{1}%
,l_{2},n_{1},n_{2},m,n,\nu,\ell,\delta_{1},\delta_{2}).$ \medskip

\noindent\textbf{Estimate} \textbf{of} $A_{6}$. Again, from (\ref{13056}),
\begin{align*}
A_{6} &  =\int_{Q}H_{m}{({U_{n}}){\varphi f}}_{n}+\int_{Q}g_{n}.\nabla
(H_{m}{({U_{n}}){\varphi)}}\\
&  =\int_{Q}H_{m}{({U_{n}})(1-{\Phi_{\delta_{1},\delta_{2}}}){R_{n,\nu,\ell}%
f}}_{n}+\int_{Q}g_{n}.\nabla(H_{m}{({U_{n}})(1-{\Phi_{\delta_{1},\delta_{2}}%
}){R_{n,\nu,\ell})+}}\omega(l_{1},l_{2},n_{1},n_{2}).
\end{align*}
Thus we can write ${A_{6}}={D}_{1}+{D}_{2}+{D}_{3}+{D}_{4}+\omega(l_{1}%
,l_{2},n_{1},n_{2}),$ where
\begin{align*}
{D}_{1} &  =\int_{Q}H{{_{m}}({U_{n}})(1-{\Phi_{\delta_{1},\delta_{2}}%
}){R_{n,\nu,\ell}}{f_{n},\qquad}D}_{2}=\int_{Q}{(1-{\Phi_{\delta_{1}%
,\delta_{2}}}){R_{n,\nu,\ell}}H_{m}^{\prime}({U_{n}}){g_{n}.}\nabla{U_{n},}}\\
&  {D}_{3}=\int_{Q}H{{_{m}}({U_{n}})(1-{\Phi_{\delta_{1},\delta_{2}}}){g_{n}%
}.\nabla{R_{n,\nu,\ell},\qquad}D}_{4}=-\int_{Q}H{{_{m}}({U_{n}}){R_{n,\nu
,\ell}g_{n}}.\nabla}{\Phi_{\delta_{1},\delta_{2}}.}%
\end{align*}
Since $\left\{  f_{n}\right\}  $ converges to $f$ weakly in $L^{1}(Q)$, and
(\ref{13057})-(\ref{13058}) hold, we get, from Remark \ref{05041},
\[
{D}_{1}=\int_{Q}{(1-{\Phi_{\delta_{1},\delta_{2}}})\left(  {{T_{\ell+k}%
}\left(  {U-}\langle{{{{{T_{k}}(U)\rangle}}_{\nu}}}\right)  -{T_{\ell-k}%
}\left(  {U-{T_{k}}\left(  U\right)  }\right)  }\right)  f+}\omega
(m,n)=\omega(m,n,\nu,\ell).
\]
We deduce from (\ref{lim3}) that ${D}_{2}=\omega(m)$. Next consider $D_{3}.$
Note that $H_{m}{(}${{$U$}}${{_{n}})=1+\omega(m),}$ and (\ref{13058}) holds,
and $\left\{  g_{n}\right\}  $ converges to $g$ strongly in $(L^{p^{\prime}%
}(Q))^{N},$ and $\langle T_{k}(U)\rangle_{\nu}$ converges to $T_{k}(U)$
strongly in $X.$ Then we obtain successively that
\begin{align*}
{D}_{3} &  =\int_{Q}{(1-{\Phi_{\delta_{1},\delta_{2}}})g.\nabla\left(
{{T_{\ell+k}}\left(  {U-\langle T_{k}(U)\rangle}_{\nu}\right)  -{T_{\ell-k}%
}\left(  {U-{T_{k}}\left(  U\right)  }\right)  }\right)  +}\omega(m,n)\\
&  =\int_{Q}{(1-{\Phi_{\delta_{1},\delta_{2}}})g.\nabla\left(  {{T_{\ell+k}%
}\left(  {U-{T_{k}}(U)}\right)  -{T_{\ell-k}}\left(  {U-{T_{k}}\left(
U\right)  }\right)  }\right)  +}\omega(m,n,\nu)\\
&  =\omega(m,n,\nu,\ell).
\end{align*}
Similarly we also get $D_{4}=\omega(m,n,\nu,\ell)$. Thus ${A_{6}}=\omega
(l_{1},l_{2},n_{1},n_{2},m,n,\nu,\ell,\delta_{1},\delta_{2}).\medskip$

\noindent\textbf{Estimate} \textbf{of} $A_{7}$. We have
\begin{align*}
\left\vert A_{7}\right\vert  &  =\left\vert \int_{Q}{{{{S^{\prime}}_{m}%
}({U_{n}})\left(  {1-{\Phi_{{\delta_{1}},{\delta_{2}}}}}\right)
{R_{n,\nu,\ell}}d\left(  {{\rho_{n,0}}-{\eta_{n,0}}}\right)  }}\right\vert
+\omega({l_{1}},{l_{2}},{n_{1}},{n_{2}})\\
&  \leq4k\int_{Q}{\left(  {1-{\Phi_{{\delta_{1}},{\delta_{2}}}}}\right)
d\left(  \rho_{n}+\eta_{n}\right)  }+\omega({l_{1}},{l_{2}},{n_{1}},{n_{2}}).
\end{align*}
From (\ref{12056}) and (\ref{12057}) we get $A_{7}=\omega(l_{1},l_{2}%
,n_{1},n_{2},m,n,\nu,\ell,\delta_{1},\delta_{2}).\medskip$

\noindent\textbf{Estimate} \textbf{of} $A_{1}+A_{2}+A_{3}$. We set
\[
J(r)={T_{\ell-k}}\left(  r{-{T_{k}}\left(  r\right)  }\right)  ,\qquad\forall
r\in\mathbb{R},
\]
and use the notations $\overline{J}{\ }${and}$\mathcal{J}$ of (\ref{lam}).
From the definitions of $\xi_{1,n_{1}},\xi_{1,n_{2}},$ we can see that
\begin{align}
A_{1}+A_{2} &  =-\int_{\Omega}J(U_{n}(T)){{\overline{H_{m}}}({U_{n}(T)}%
)dx}-\int_{\Omega}T_{\ell+k}(u_{0,n}-z_{\nu}){{\overline{H_{m}}}(}%
u_{0,n})dx+\omega(l_{1},l_{2},n_{1},n_{2})\nonumber\\
&  =-\int_{\Omega}J(U_{n}(T))U_{n}(T)dx-\int_{\Omega}T_{\ell+k}(u_{0,n}%
-z_{\nu})u_{0,n}dx+\omega(l_{1},l_{2},n_{1},n_{2},m),\label{a1a2}%
\end{align}
where $z_{\nu}=\langle T_{k}(U)\rangle_{\nu}(0).$ We can write $A_{3}%
=F_{1}+F_{2},$ where
\begin{align*}
\text{ }F_{1} &  =-\int_{Q}{{{\left(  {{\xi_{n_{1}}}(1-{\Phi_{\delta
_{1},\delta_{2}}}){{\left[  {{T_{\ell+k}}\left(  {U_{n}-\langle T_{k}%
(U)\rangle}_{\nu}\right)  }\right]  }_{{l_{1}}}}}\right)  }_{t}\overline
{H_{m}}}({U_{n}}),}\\
F_{2} &  =\int_{Q}{{{\left(  {{\xi_{n_{2}}}(1-{\Phi_{\delta_{1},\delta_{2}}%
}){{\left[  {{T_{\ell-k}}\left(  {U_{n}-{T_{k}}\left(  {U_{n})}\right)
}\right)  }\right]  }_{{-l_{2}}}}}\right)  }_{t}\overline{H_{m}}}({U_{n}}).}%
\end{align*}
\textbf{Estimate} \textbf{of} $F_{2}$. We write $F_{2}=G_{1}+G_{2}+G_{3},$
with
\begin{align*}
G_{1} &  =-\int_{Q}{{{\left(  {{\Phi_{\delta_{1},\delta_{2}}}}\right)  }_{t}%
}{\xi_{n_{2}}}{{\left[  {{T_{\ell-k}}\left(  {U_{n}-{T_{k}}\left(
U_{n}\right)  }\right)  }\right]  }_{{-l_{2}}}\overline{H_{m}}}({U_{n}}),}\\
G_{2} &  =\int_{Q}{(1-{\Phi_{\delta_{1},\delta_{2}}}){{\left(  {{\xi_{n_{2}}}%
}\right)  }_{t}}{{\left[  {{T_{\ell-k}}\left(  {U_{n}-{T_{k}}\left(
U_{n}\right)  }\right)  }\right]  }_{{-l_{2}}}}\overline{H_{m}}(U_{n}),}\\
G_{3} &  =\int_{Q}{{\xi_{n_{2}}}(1-{\Phi_{\delta_{1},\delta_{2}}}){{\left(
{{{\left[  {{T_{\ell-k}}\left(  {U_{n}-{T_{k}}\left(  {U_{n}}\right)
}\right)  }\right]  }_{{-l_{2}}}}}\right)  }_{t}}\overline{H_{m}}(U_{n}).}%
\end{align*}
We find easily that
\[
{G}_{1}=-\int_{Q}{{{\left(  {{\Phi_{\delta_{1},\delta_{2}}}}\right)  }_{t}%
J}(U_{n})U_{n}+}\omega(l_{1},l_{2},n_{1},n_{2},m),
\]%
\[
{G}_{2}=\int_{Q}{(1-{\Phi_{\delta_{1},\delta_{2}}}){{\left(  {{\xi_{n_{2}}}%
}\right)  }_{t}}}J(U_{n}){{\overline{H_{m}}}({U_{n}})+}\omega({l_{1},l_{2}%
})=\int_{\Omega}J{(u_{0,n})u_{0,n}dx+}\omega(l_{1},l_{2},n_{1},n_{2},m).
\]
Next consider $G_{3}.$ Setting $b={{\overline{H_{m}}}(}${{$U$}}${{_{n}}){,}}$
there holds from (\ref{tkp}) and (\ref{222}),
\[
((\left[  J{{(b)}}\right]  _{-{l_{2}}})_{t}b)(.,t)=\frac{{{b(.,t)}}}{l_{2}%
}(J{{(b)(.,t)-}}J{{{{(b)(.,t-l}}_{2}{{)}}).}}%
\]
Hence
\[
{\left(  {{{\left[  {{T_{\ell-k}}\left(  {{U_{n}}-{T_{k}}\left(  {{U_{n}}%
}\right)  }\right)  }\right]  }_{-{l_{2}}}}}\right)  _{t}\overline{H_{m}}%
}({U_{n}})\geq{\left(  {{{\left[  \mathcal{J}{({\overline{H_{m}}}({U_{n}}%
))}\right]  }_{-{l_{2}}}}}\right)  _{t}=\left(  {{{\left[  \mathcal{J}%
{({U_{n}})}\right]  }_{-{l_{2}}}}}\right)  _{t},}%
\]
since $\mathcal{J}$ is constant in $\left\{  \left\vert r\right\vert \geq
m+\ell+2k\right\}  .$ Integrating by parts in $G_{3},$ we find
\begin{align*}
G_{3} &  \geq\int_{Q}{{\xi_{2,n_{2}}}(1-{\Phi_{\delta_{1},\delta_{2}}%
}){{\left(  {{{\left[  \mathcal{J}{({U_{n}})}\right]  }_{{-l_{2}}}}}\right)
}_{t}}}=-\int_{Q}{{{\left(  {{\xi_{2,n_{2}}}(1-{\Phi_{\delta_{1},\delta_{2}}%
})}\right)  }_{t}}{{\left[  \mathcal{J}{({U_{n}})}\right]  }_{{-l_{2}}}}}%
+\int_{\Omega}{{\xi_{2,n_{2}}}}(T){{{\left[  \mathcal{J}{({U_{n}})}\right]
}_{{-l_{2}}}}(T)dx}\\
&  =-\int_{Q}{{{\left(  {{\xi_{2,n_{2}}}}\right)  }_{t}}(1-{\Phi_{\delta
_{1},\delta_{2}}})}\mathcal{J}{({U_{n}})}+\int_{Q}{{\xi_{2,n_{2}}{\left(
{{\Phi_{\delta_{1},\delta_{2}}}}\right)  }_{t}}}\mathcal{J}{({U_{n}})}%
+\int_{\Omega}{{\xi_{2,n_{2}}}}(T)\mathcal{J}{({U_{n}}(T))dx+}\omega
({l_{1},l_{2}})\\
&  =-\int_{\Omega}\mathcal{J}{({u_{0,n}})dx+\int_{Q}{{{\left(  {{\Phi
_{\delta_{1},\delta_{2}}}}\right)  }_{t}}\mathcal{J}{({U_{n}})}+}\int_{\Omega
}\mathcal{J}{({U_{n}}(T))dx}+}\omega(l_{1},l_{2},n_{1},n_{2}).
\end{align*}
Therefore, since $\mathcal{J}(${{$U$}}${{_{n}}})-J(${{$U$}}${{_{n}}})${{$U$}%
}${{_{n}}}=-{\overline{J}(}${{$U$}}${{_{n}})}$ and ${\overline{J}(u_{0,n}%
)=}J{(u_{0,n})u_{0,n}-}\mathcal{J}{(u_{0,n}),}$ we obtain
\begin{equation}
{F}_{2}\geq\int_{\Omega}{\overline{J}(u_{0,n})dx}\text{ }-\int_{Q}{{{\left(
{{\Phi_{\delta_{1},\delta_{2}}}}\right)  }_{t}}{\overline{J}}({U_{n}})}%
+\int_{\Omega}\mathcal{J}{(U_{n}(T))dx+}\omega(l_{1},l_{2},n_{1}%
,n_{2},m).\label{f2}%
\end{equation}
\textbf{Estimate of} $F_{1}.$ Since $m>\ell+2k,$ there holds ${{T_{\ell+k}%
}\left(  {{U_{n}}-}\langle{{{{{T_{k}}(U)\rangle}}_{\nu}}}\right)  ={T_{\ell
+k}}\left(  {{\overline{H_{m}}}({U_{n}})-}\langle{{{{{T_{k}}({\overline{H_{m}%
}}(U))\rangle}}_{\nu}}}\right)  }$ on supp${{\overline{H_{m}}}(}${{$U$}%
}${{_{n}}).}$ Hence we can write $F_{1}=L_{1}+L_{2},$ with
\begin{align*}
L_{1} &  =-\int_{Q}{{{\left(  {{\xi_{1,n_{1}}}(1-{\Phi_{\delta_{1},\delta_{2}%
}}){{\left[  {{T_{\ell+k}}\left(  {{\overline{H_{m}}}({U_{n}})-}%
\langle{{{{{T_{k}}({\overline{H_{m}}}(U))\rangle}}_{\nu}}}\right)  }\right]
}_{{l_{1}}}}}\right)  }_{t}}\left(  {{\overline{H_{m}}}({U_{n}})-}%
\langle{{{{{T_{k}}({\overline{H_{m}}}(U)\rangle}}_{\nu}}}\right)  }\\
L_{2} &  =-\int_{Q}{{{\left(  {{\xi_{1,n_{1}}}(1-{\Phi_{\delta_{1},\delta_{2}%
}}){{\left[  {{T_{\ell+k}}\left(  {{\overline{H_{m}}}({U_{n}})-}%
\langle{{{{{T_{k}}({\overline{H_{m}}}(U))\rangle}}_{\nu}}}\right)  }\right]
}_{{l_{1}}}}}\right)  }_{t}}}\langle{{{{{T_{k}}({\overline{H_{m}}}(U))\rangle
}}_{\nu}.}}%
\end{align*}
Integrating by parts we have, by definition of the Landes-time approximation,
\begin{align}
L_{2} &  =\int_{Q}{{\xi_{1,n_{1}}}(1-{\Phi_{\delta_{1},\delta_{2}}}){{\left[
{{T_{\ell+k}}\left(  {{\overline{H_{m}}}({U_{n}})-}\langle{{{{{T_{k}%
}({\overline{H_{m}}}(U))\rangle}}_{\nu}}}\right)  }\right]  }_{{l_{1}}}%
}{{\left(  \langle{{{{{T_{k}}({\overline{H_{m}}}(U))\rangle}}_{\nu}}}\right)
}_{t}}}\nonumber\\
&  +\int_{\Omega}{{\xi_{1,n_{1}}}}(0){{{\left[  {{T_{\ell+k}}\left(
{{\overline{H_{m}}}({U_{n}})-}\langle{{{{{T_{k}}({\overline{H_{m}}}%
(U))\rangle}}_{\nu}}}\right)  }\right]  }_{{l_{1}}}}(0)\langle{{{{{T_{k}%
}({\overline{H_{m}}}(U))\rangle}}_{\nu}}}(0)dx}\nonumber\\
&  =\nu\int_{Q}{(1-{\Phi_{\delta_{1},\delta_{2}}}){T_{\ell+k}}\left(  {U_{n}%
-}\langle{{{{{T_{k}}(U)\rangle}}_{\nu}}}\right)  \left(  {{T_{k}}(U)-}%
\langle{{{{{T_{k}}(U)\rangle}}_{\nu}}}\right)  }+\int_{\Omega}{{T_{\ell+k}%
}\left(  {{u_{0,n}}-{z_{\nu}}}\right)  {z_{\nu}dx+}}\omega(l_{1},l_{2}%
,n_{1},n_{2}).\label{a2}%
\end{align}
We decompose $L_{1}$ into $L_{1}=K_{1}+K_{2}+K_{3},$ where
\begin{align*}
K_{1} &  =-\int_{Q}{{{\left(  {{\xi_{1,n_{1}}}}\right)  }_{t}}(1-{\Phi
_{\delta_{1},\delta_{2}}}){{\left[  {{T_{\ell+k}}\left(  {{\overline{H_{m}}%
}({U_{n}})-}\langle{{{{{T_{k}}({\overline{H_{m}}}(U))\rangle}}_{\nu}}}\right)
}\right]  }_{{l_{1}}}}\left(  {{\overline{H_{m}}}({U_{n}})-}\langle{{{{{T_{k}%
}({\overline{H_{m}}}(U))\rangle}}_{\nu}}}\right)  }\\
K_{2} &  =\int_{Q}{{\xi_{1,n_{1}}{\left(  {{\Phi_{\delta_{1},\delta_{2}}}%
}\right)  }_{t}}{{\left[  {{T_{\ell+k}}\left(  {{\overline{H_{m}}}({U_{n}}%
)-}\langle{{{{{T_{k}}({\overline{H_{m}}}(U))\rangle}}_{\nu}}}\right)
}\right]  }_{{l_{1}}}}\left(  {{\overline{H_{m}}}({U_{n}})-}\langle{{{{{T_{k}%
}({\overline{H_{m}}}(U))\rangle}}_{\nu}}}\right)  }\\
K_{3} &  =-\int_{Q}{{\xi_{1,n_{1}}}(1-{\Phi_{\delta_{1},\delta_{2}}}){{\left(
{{{\left[  {{T_{\ell+k}}\left(  {{\overline{H_{m}}}({U_{n}})-}\langle
{{{{{T_{k}}({\overline{H_{m}}}(U))\rangle}}_{\nu}}}\right)  }\right]
}_{{l_{1}}}}}\right)  }_{t}}\left(  {{\overline{H_{m}}}({U_{n}})-}%
\langle{{{{{T_{k}}({\overline{H_{m}}}(U)\rangle}}_{\nu}}}\right)  .}%
\end{align*}
Then we check easily that
\[
K_{1}=\int_{\Omega}{{T_{\ell+k}}\left(  {{U_{n}}-}\langle{{{{{T_{k}}%
(U)\rangle}}_{\nu}}}\right)  (T)\left(  {{U_{n}}-}\langle{{{{{T_{k}}%
(U)\rangle}}_{\nu}}}\right)  (T)dx+}\omega(l_{1},l_{2},n_{1},n_{2},m),
\]%
\[
K_{2}=\int_{Q}{{{\left(  {{\Phi_{\delta_{1},\delta_{2}}}}\right)  }_{t}%
}{T_{\ell+k}}\left(  {{U_{n}}-}\langle{{{{{T_{k}}(U)\rangle}}_{\nu}}}\right)
\left(  {{U_{n}}-}\langle{{{{{T_{k}}(U)\rangle}}_{\nu}}}\right)  +}%
\omega(l_{1},l_{2},n_{1},n_{2},m).
\]
Next consider $K_{3}.$ Here we use the function $\mathcal{T}_{k}$ defined at
(\ref{tkp}). We set $b={{\overline{H_{m}}}(}${{$U$}}${{_{n}})-}\langle
{{{{{T_{k}}({\overline{H_{m}}}(}}}}${{{{$U$}}}}${{{{))\rangle}}_{\nu}.}}$
Hence from (\ref{222}),
\begin{align*}
((\left[  {{T_{\ell+k}(b)}}\right]  _{{l_{1}}})_{t}b)(.,t) &  =\frac
{{{b(.,t)}}}{l_{1}}({{T_{\ell+k}(b)(.,t+l}}_{1}{{)-{{T_{\ell+k}(b)(.,t)}})}}\\
{} &  {{\leq}}\frac{1}{l_{1}}(\mathcal{T}_{\ell+k}(b)({{(.,t+l}}_{1}{{))}%
}-\mathcal{T}_{\ell+k}(b)(.,t))=(\left[  \mathcal{T}_{\ell+k}(b)\right]
_{l_{1}})_{t}.
\end{align*}
Thus
\[
({{{{\left[  {{T_{\ell+k}}\left(  {{\overline{H_{m}}}({U_{n}})-}%
\langle{{{{{T_{k}}({\overline{H_{m}}}(U))\rangle}}_{\nu}}}\right)  }\right]
}_{{l_{1}}})}}_{t}}\left(  {{\overline{H_{m}}}({U_{n}})-}\langle{{{{{T_{k}%
}({\overline{H_{m}}}(U))\rangle}}_{\nu}}}\right)  \leq({{{{\left[
\mathcal{T}_{\ell+k}{{({U_{n}}-}\langle{{{{{T_{k}}(U)\rangle}}_{\nu}}}%
}\right]  }_{{l_{1}}})}}_{t}.}%
\]
Then
\begin{align*}
{K}_{3} &  \geq-\int_{Q}{{\xi_{1,n_{1}}}(1-{\Phi_{\delta_{1},\delta_{2}}}%
)}({{{{{{\left[  \mathcal{T}{_{\ell+k}\left(  {{U_{n}}-}\langle{{{{{T_{k}%
}(U)\rangle}}_{\nu}}}\right)  }\right]  }_{{l_{1}}})}}}_{t}}}\\
&  =\int_{Q}{{{\left(  {{\xi_{1,n_{1}}}}\right)  }_{t}}(1-{\Phi_{\delta
_{1},\delta_{2}}}){{\left[  \mathcal{T}{_{\ell+k}\left(  {{U_{n}}-}%
\langle{{{{{T_{k}}(U)\rangle}}_{\nu}}}\right)  }\right]  }_{{l_{1}}}}}%
-\int_{Q}{{\xi_{1,n_{1}}{\left(  {{\Phi_{\delta_{1},\delta_{2}}}}\right)
}_{t}}{{\left[  \mathcal{T}{_{\ell+k}\left(  {{U_{n}}-}\langle{{{{{T_{k}%
}(U)\rangle}}_{\nu}}}\right)  }\right]  }_{{l_{1}}}}}\\
&  +\int_{\Omega}{{\xi_{1,n_{1}}}}(0){{{\left[  \mathcal{T}{_{\ell+k}\left(
{{U_{n}}-}\langle{{{{{T_{k}}(U)\rangle}}_{\nu}}}\right)  }\right]  }_{{l_{1}}%
}}(0)dx}\\
&  =-\int_{\Omega}\mathcal{T}{_{\ell+k}\left(  {{U_{n}(T)}-\langle{{{{{T_{k}%
}(U)\rangle}}_{\nu}{{(T)}}}}}\right)  dx}-\int_{Q}{{{\left(  {{\Phi
_{\delta_{1},\delta_{2}}}}\right)  }_{t}}}\mathcal{T}{_{\ell+k}\left(
{{U_{n}}-}\langle{{{{{T_{k}}(U)\rangle}}_{\nu}}}\right)  }\\
&  +\int_{\Omega}\mathcal{T}{{_{\ell+k}}\left(  {{u_{0,n}}-{z_{\nu}}}\right)
dx+}\omega({l_{1}},{l_{2}},{n_{1}},{n_{2}}).
\end{align*}
We find by addition, since $T_{\ell+k}(r)-\mathcal{T}{{_{\ell+k}%
(r)={\overline{T}}_{\ell+k}(r)}}$ for any $r\in\mathbb{R},$
\begin{align}
{L}_{1} &  \geq\int_{\Omega}\mathcal{T}{{_{\ell+k}}\left(  {{u_{0,n}}-{z_{\nu
}}}\right)  dx}+\int_{\Omega}{{{\overline{T}}_{\ell+k}}\left(  {{U_{n}%
}(T)-\langle{{{{{T_{k}}(U)\rangle}}_{\nu}}}(T)}\right)  dx}\nonumber\\
&  +\int_{Q}{{{\left(  {{\Phi_{{\delta_{1}},{\delta_{2}}}}}\right)  }_{t}%
}{{\overline{T}}_{\ell+k}}\left(  {{U_{n}}-}\langle{{{{{T_{k}}(U)\rangle}%
}_{\nu}}}\right)  +}\omega({l_{1}},{l_{2}},{n_{1}},{n_{2}},m).\label{a1}%
\end{align}
We deduce from (\ref{a1}), (\ref{a2}), (\ref{f2}),
\begin{align}
{A}_{3} &  \geq\int_{\Omega}\overline{J}{({u_{0,n}})dx}+\int_{\Omega
}\mathcal{T}{{_{\ell+k}}\left(  {{u_{0,n}}-{z_{\nu}}}\right)  dx}+\int%
_{\Omega}{{T_{\ell+k}}\left(  {{u_{0,n}}-{z_{\nu}}}\right)  {z_{\nu}dx}%
}\label{a3}\\
&  +\int_{\Omega}{{{\overline{T}}_{\ell+k}}\left(  {{U_{n}}(T)-}%
\langle{{{{{T_{k}}(U)\rangle}}_{\nu}}(T)}\right)  dx}+\int_{\Omega}%
\mathcal{J}{({U_{n}}(T))dx}+\int_{Q}{{{\left(  {{\Phi_{{\delta_{1}}%
,{\delta_{2}}}}}\right)  }_{t}}\left(  {{{\overline{T}}_{\ell+k}}\left(
{{U_{n}}-}\langle{{{{{T_{k}}(U)\rangle}}_{\nu}}}\right)  -\overline{J}({U_{n}%
})}\right)  }\nonumber\\
&  +\nu\int_{Q}{(1-{\Phi_{{\delta_{1}},{\delta_{2}}}}){T_{\ell+k}}\left(
{{U_{n}}-}\langle{{{{{T_{k}}(U)\rangle}}_{\nu}}}\right)  \left(  {{T_{k}}%
(U)-}\langle{{{{{T_{k}}(U)\rangle}}_{\nu}}}\right)  +}\omega({l_{1}},{l_{2}%
},{n_{1}},{n_{2}},m).\nonumber
\end{align}
Next we add (\ref{a1a2}) and (\ref{a3}). Note that $\mathcal{J}{(}${{$U$}%
}${{_{n}}(T))-J(}${{$U$}}${{_{n}}(T))}${{$U$}}${{_{n}}(T)=-\overline{J}(}%
${{$U$}}${{_{n}}(T)),}$ and also
\[
\mathcal{T}{{_{\ell+k}}\left(  {{u_{0,n}}-{z_{\nu}}}\right)  -{T_{\ell+k}%
}\left(  {{u_{0,n}}-{z_{\nu}}}\right)  ({z_{\nu}-{u_{0,n}})=-{\overline{T}%
}_{\ell+k}}\left(  {{u_{0,n}}-{z_{\nu}}}\right)  .}%
\]
Then we find
\begin{align*}
A_{1}+A_{2}+A_{3} &  \geq\int_{\Omega}{\left(  \overline{J}{({u_{0,n}%
})-{{\overline{T}}_{\ell+k}}\left(  {{u_{0,n}}-{z_{\nu}}}\right)  }\right)
dx}+\int_{\Omega}{\left(  {{{\overline{T}}_{\ell+k}}\left(  {{U_{n}%
}(T)-\langle{{{{{T_{k}}(U)\rangle}}_{\nu}}}(T)}\right)  -\overline{J}({U_{n}%
}(T))}\right)  dx}\\
&  +\int_{Q}{{{\left(  {{\Phi_{{\delta_{1}},{\delta_{2}}}}}\right)  }_{t}%
}\left(  {{{\overline{T}}_{\ell+k}}\left(  {{U_{n}}-}\langle{{{{{T_{k}%
}(U)\rangle}}_{\nu}}}\right)  -\overline{J}({U_{n}})}\right)  }\\
&  +\nu\int_{Q}{(1-{\Phi_{{\delta_{1}},{\delta_{2}}}}){T_{\ell+k}}\left(
{{U_{n}}-}\langle{{{{{T_{k}}(U)\rangle}}_{\nu}}}\right)  \left(  {{T_{k}}%
(U)-}\langle{{{{{T_{k}}(U)\rangle}}_{\nu}}}\right)  +}\omega({l_{1}},{l_{2}%
},{n_{1}},{n_{2}},m).
\end{align*}
Notice that ${{{{\overline{T}}_{\ell+k}}\left(  r{-}s\right)  -\overline
{J}(r)}}${$\geq$}${0}$ for any $r,s\in\mathbb{R}$ such that $\left\vert
s\right\vert \leq k;$ thus
\[
\int_{\Omega}{\left(  {{{\overline{T}}_{\ell+k}}\left(  {{U_{n}}(T)-}%
\langle{{{{{T_{k}}(U)\rangle}}_{\nu}}(T)}\right)  -\overline{J}({U_{n}}%
(T))}\right)  dx\geq0.}%
\]
And $\left\{  {{u_{0,n}}}\right\}  $ converges to $u_{0}$ in $L^{1}(\Omega)$
and $\left\{  U_{n}\right\}  $ converges to $U$ in $L^{1}(Q)$ from Proposition
\ref{mun}. Thus we obtain%
\[%
\begin{array}
[c]{c}%
A_{1}+A_{2}+A_{3}\geq\int_{\Omega}{\left(  \overline{J}{({u_{0}}%
)-{{\overline{T}}_{\ell+k}}\left(  {{u_{0}}-{z_{\nu}}}\right)  }\right)
dx}+\int_{Q}{{{\left(  {{\Phi_{{\delta_{1}},{\delta_{2}}}}}\right)  }_{t}%
}\left(  {{{\overline{T}}_{\ell+k}}\left(  {U-}\langle{{{{{T_{k}}(U)\rangle}%
}_{\nu}}}\right)  -\overline{J}(U)}\right)  }\\
\\
+\nu\int_{Q}{(1-{\Phi_{{\delta_{1}},{\delta_{2}}}}){T_{\ell+k}}\left(
{U-}\langle{{{{{T_{k}}(U)\rangle}}_{\nu}}}\right)  \left(  {{T_{k}}%
(U)-}\langle{{{{{T_{k}}(U)\rangle}}_{\nu}}}\right)  +}\omega({l_{1}},{l_{2}%
},{n_{1}},{n_{2}},m,n).
\end{array}
\]
Moreover ${{T_{\ell+k}}\left(  r{-s}\right)  \left(  {{T_{k}}(r)-s}\right)  }%
${$\geq$}${0}$ for any $r,s\in\mathbb{R}$ such that $\left\vert s\right\vert
\leq k,$ hence
\begin{align*}
A_{1}+A_{2}+A_{3} &  \geq\int_{\Omega}{\left(  \overline{J}{({u_{0}%
})-{{\overline{T}}_{\ell+k}}\left(  {{u_{0}}-{z_{\nu}}}\right)  }\right)
dx}+\int_{Q}{{{\left(  {{\Phi_{{\delta_{1}},{\delta_{2}}}}}\right)  }_{t}%
}\left(  {{{\overline{T}}_{\ell+k}}\left(  {U-}\langle{{{{{T_{k}}(U)\rangle}%
}_{\nu}}}\right)  -\overline{J}(U)}\right)  }\\
& \\
&  {+}\omega({l_{1}},{l_{2}},{n_{1}},{n_{2}},m,n).
\end{align*}
As $\nu\rightarrow\infty,$ $\left\{  z_{\nu}\right\}  $ converges to
$T_{k}(u_{0}),$ $a.e.$ in $\Omega$, thus we get%
\begin{align*}
A_{1}+A_{2}+A_{3} &  \geq\int_{\Omega}{\left(  \overline{J}{({u_{0}%
})-{{\overline{T}}_{\ell+k}}\left(  {{u_{0}}-{T_{k}}({u_{0}})}\right)
}\right)  dx}+\int_{Q}{{{\left(  {{\Phi_{{\delta_{1}},{\delta_{2}}}}}\right)
}_{t}}\left(  {{{\overline{T}}_{\ell+k}}\left(  {U-{T_{k}}(U)}\right)
-\overline{J}(U)}\right)  }\\
&  +\omega({l_{1}},{l_{2}},{n_{1}},{n_{2}},m,n,\nu).
\end{align*}
Finally $\left\vert {{{\overline{T}}_{\ell+k}}\left(  r{-{T_{k}}(r)}\right)
-\overline{J}(r)}\right\vert \leq2k|r|{\chi_{\left\{  \left\vert r\right\vert
\geq\ell\right\}  }}$ for any $r\in\mathbb{R},$ thus
\[
A_{1}+A_{2}+A_{3}\geq\omega({l_{1}},{l_{2}},{n_{1}},{n_{2}},m,n,\nu,\ell).
\]
Combining all the estimates, we obtain $I_{2}\leq\omega(l_{1},l_{2}%
,n_{1},n_{2},m,n,\nu,\ell,\delta_{1},\delta_{2}),$ which implies
(\ref{120510}), since $I_{2}$ does not depend on $l_{1},l_{2},n_{1}%
,n_{2},m,\ell.$\medskip
\end{proof}

Next we conclude the proof of Theorem \ref{sta}:

\begin{lemma}
\label{concl} The function $u$ is a R-solution of (\ref{pmu}).
\end{lemma}

\begin{proof}
(i) First show that $u$ satisfies (\ref{renor}). Here we proceed as in
\cite{Pe08}. Let $\varphi\in X\cap L^{\infty}(Q)$ such $\varphi_{t}\in
X^{\prime}+L^{1}(Q),$ $\varphi(.,T)=0,$ and $S\in W^{2,\infty}(\mathbb{R})$,
such that $S^{\prime}$ has compact support on $\mathbb{R}$, $S(0)=0$. Let
$M>0$ such that supp$S^{\prime}\subset\lbrack-M,M]$. Taking successively
$(\varphi,S)$ and $(\varphi\psi_{\delta}^{\pm},S)$ as test functions in
(\ref{renor}) applied to $u_{n}$, we can write
\[
A_{1}+A_{2}+A_{3}+A_{4}=A_{5}+A_{6}+A_{7},\qquad A_{2,\delta,\pm}%
+A_{3,\delta,\pm}+A_{4,\delta,\pm}=A_{5,\delta,\pm}+A_{6,\delta,\pm
}+A_{7,\delta,\pm},
\]
where
\begin{align*}
A_{1} &  =-\int_{\Omega}{\varphi(0)S({u_{0,n}})dx,\quad}A_{2}=-\int%
_{Q}{{\varphi_{t}}S({U_{n}}),\quad}A_{2,\delta,\pm}=-\int_{Q}(\varphi
\psi_{\delta}^{\pm}){{_{t}}S({U_{n}}),}\\
A_{3} &  =\int_{Q}{S^{\prime}({U_{n}})A(x,t,\nabla u_{n}).\nabla\varphi,\quad
}A_{3,\delta,\pm}=\int_{Q}{S^{\prime}({U_{n}})A(x,t,\nabla u_{n}%
).\nabla(\varphi\psi_{\delta}^{\pm}),}%
\end{align*}%
\begin{align*}
A_{4} &  =\int_{Q}{S^{\prime\prime}({U_{n}})\varphi A(x,t,\nabla u_{n}).\nabla
U_{n},\quad}A_{4,\delta,\pm}=\int_{Q}{S^{\prime\prime}({U_{n}})\varphi
\psi_{\delta}^{\pm}A(x,t,\nabla u_{n}).\nabla U_{n},}\\
A_{5} &  =\int_{Q}{S^{\prime}(}U{{_{n}})\varphi d{\widehat{\lambda_{n,0}%
},\quad A}}_{6}=\int_{Q}{S^{\prime}({U_{n}})\varphi d{\rho_{n,0},}\quad}%
A_{7}=-\int_{Q}{S^{\prime}(}U{{_{n}})\varphi d{\eta_{n,0},}}\\
A_{5,\delta,\pm} &  =\int_{Q}{S^{\prime}(}U{{_{n}})\varphi\psi_{\delta}^{\pm
}d{\widehat{\lambda_{n,0}},{\quad}A}}_{6,\delta,\pm}=\int_{Q}{S^{\prime
}({U_{n}})\varphi\psi_{\delta}^{\pm}d{\rho_{n,0},}\quad}A_{7,\delta,\pm}%
=-\int_{Q}{S^{\prime}(}U{{_{n}})\varphi\psi_{\delta}^{\pm}d{\eta_{n,0}.}}%
\end{align*}

\noindent\ Since $\left\{  u_{0,n}\right\}  $ converges to $u_{0}$ in
$L^{1}(\Omega),$ and $\left\{  S({U_{n}})\right\}  $ converges to $S(U),$
strongly in $X$ and weak-$^{\ast}$ in $L^{\infty}(Q),$ there holds, from
(\ref{12054}),
\[
A_{1}=-\int_{\Omega}{\varphi(0)S({u_{0}})dx}+\omega(n),\quad A_{2}=-\int%
_{Q}{{\varphi_{t}}S(U)}+\omega(n),\quad A_{2,\delta,\psi_{\delta}^{\pm}%
}=\omega(n,\delta).
\]
\newline Moreover $T_{M}(${$U$}${_{n}})$ converges to $T_{M}(U)$, then
$T_{M}(${$U$}${_{n}})+h_{n}$ converges to $T_{k}(U)+h$ strongly in $X$, thus%
\begin{align*}
A_{3} &  =\int_{Q}{S^{\prime}({U_{n}})A(x,t,\nabla\left(  {{T_{M}}\left(
{{U_{n}}}\right)  +{h_{n}}}\right)  ).\nabla\varphi}=\int_{Q}{S^{\prime
}(U)A(x,t,\nabla\left(  {{T_{M}}\left(  {{U}}\right)  +{h}}\right)
).\nabla\varphi}+\omega(n)\\
&  =\int_{Q}{S^{\prime}(U)A(x,t,\nabla u).\nabla\varphi}+\omega(n);
\end{align*}
and
\begin{align*}
A_{4} &  =\int_{Q}{S^{\prime\prime}({U_{n}})\varphi A(x,t,\nabla\left(
{{T_{M}}\left(  {{U_{n}}}\right)  +{h_{n}}}\right)  ).\nabla{T_{M}}\left(
{{U_{n}}}\right)  }\\
&  =\int_{Q}{S^{\prime\prime}(U)\varphi A(x,t,\nabla\left(  {{T_{M}}\left(
{{U}}\right)  +{h}}\right)  ).\nabla{T_{M}}\left(  U\right)  }+\omega
(n)=\int_{Q}{S^{\prime\prime}(U)\varphi A(x,t,\nabla u).\nabla U}+\omega(n).
\end{align*}
In the same way, since ${\psi_{\delta}^{\pm}}$ converges to $0$ in $X,$
\begin{align*}
A_{3,\delta,\pm} &  =\int_{Q}{S^{\prime}(U)A(x,t,\nabla u).\nabla(\varphi
\psi_{\delta}^{\pm})}+\omega(n)=\omega(n,\delta),\\
A_{4,\delta,\pm} &  =\int_{Q}{S^{\prime\prime}(U)\varphi\psi_{\delta}^{\pm
}A(x,t,\nabla u).\nabla U}+\omega(n)=\omega(n,\delta).
\end{align*}
And $\left\{  {{{g_{n}}}}\right\}  $ strongly converges to $g$ in
$(L^{p^{\prime}}(\Omega))^{N},$ thus
\begin{align*}
A_{5} &  =\int_{Q}{{S^{\prime}({U_{n}})\varphi{f_{n}}}+}\int_{Q}{{S^{\prime
}({U_{n}}){g_{n}.}\nabla\varphi}+}\int_{Q}{{S}}^{\prime\prime}{{({U_{n}%
})\varphi{g_{n}.}\nabla{T_{M}}({U_{n}})}}\\
&  =\int_{Q}{{S^{\prime}(U)\varphi f}+}\int_{Q}{{S^{\prime}(U)g.\nabla\varphi
}+}\int_{Q}{{S}}^{\prime\prime}{{(U)\varphi g.\nabla{T_{M}}(U)}+\omega(n)}\\
&  {=}\int_{Q}{{S^{\prime}(U)\varphi}d{\widehat{\mu_{0}}}+\omega(n).}%
\end{align*}
Now $A_{5,\delta,\pm}{=}\int_{Q}{{S^{\prime}(}}${{$U$}}${{)\varphi}%
\psi_{\delta}^{\pm}d{\widehat{\lambda_{n,0}}}+\omega(n)=}\omega(n,\delta).$
Then $A_{6,\delta,\pm}+A_{7,\delta,\pm}=\omega(n,\delta)$. From (\ref{12054})
we verify that $A_{7,\delta,+}=\omega(n,\delta)$ and $A_{6,\delta,-}%
=\omega(n,\delta)$. Moreover, from (\ref{muno}) and (\ref{12054}), we find
\[
\left\vert A_{6}-A_{6,\delta,+}\right\vert \leq\int_{Q}\left\vert {S^{\prime
}({U_{n}})\varphi}\right\vert {(1-\psi_{\delta}^{+})d{\rho_{n,0}}}%
\leq{\left\Vert S\right\Vert _{{W^{2,\infty}}(\mathbb{R})}}{\left\Vert
\varphi\right\Vert _{{L^{\infty}}({Q})}}\int_{Q}{(1-\psi_{\delta}^{+}%
)d{\rho_{n}}}=\omega(n,\delta).
\]
Similarly we also have $\left\vert A_{7}-A_{7,\delta,-}\right\vert \leq
\omega(n,\delta)$. Hence $A_{6}=\omega(n)$ and $A_{7}=\omega(n).$ Therefore,
we finally obtain (\ref{renor}):
\begin{equation}
-\int_{\Omega}{\varphi(0)S({u_{0}})dx}-\int_{Q}{{\varphi_{t}}S(U)+}\int%
_{Q}{S^{\prime}(U)A(x,t,\nabla u).\nabla\varphi+}\int_{Q}{S^{\prime\prime
}(U)\varphi A(x,t,\nabla u).\nabla U=}\int_{Q}{{S^{\prime}(U)\varphi
}d{\widehat{\mu_{0}}.}}\label{renor5}%
\end{equation}
\medskip

(ii) Next, we prove (\ref{renor2}) and (\ref{renor3}). We take $\varphi\in
C_{c}^{\infty}(Q)$ and take $({(1-\psi_{\delta}^{-})\varphi,}\overline{H_{m}%
})$ as test functions in (\ref{renor5}), with $\overline{H_{m}}$ as in
(\ref{Hm}). We can write $D_{1,m}+D_{2,m}=D_{3,m}+D_{4,m}+D_{5,m},$ where
\begin{equation}%
\begin{array}
[c]{l}%
D_{1,m}=-\int\limits_{Q}{{{\left(  {(1-\psi_{\delta}^{-})\varphi}\right)
}_{t}}\overline{H_{m}}(U)},\quad\quad D_{2,m}=\int\limits_{Q}H_{m}%
{(}U{)A(x,t,\nabla u).\nabla\left(  {(1-\psi_{\delta}^{-})\varphi}\right)
},\\
\\
D_{3,m}=\int\limits_{Q}H_{m}{(}U{)(1-\psi_{\delta}^{-})\varphi d{\widehat{\mu
_{0}},}}\quad\quad D_{4,m}=\frac{1}{m}\int\limits_{m\leq U\leq2m}%
{(1-\psi_{\delta}^{-})\varphi{A(x,t,\nabla u).\nabla{U},}}\\
\\
D_{5,m}=-\frac{1}{m}\int\limits_{-2m\leq U\leq-m}{(1-\psi_{\delta}^{-}%
)\varphi{A(x,t,\nabla u)\nabla U.}}%
\end{array}
\label{april281}%
\end{equation}
Taking the same test functions in (\ref{renor}) applied to $u_{n},$ there
holds $D_{1,m}^{n}+D_{2,m}^{n}=D_{3,m}^{n}+D_{4,m}^{n}+D_{5,m}^{n}$, where%
\begin{equation}%
\begin{array}
[c]{l}%
D_{1,m}^{n}=-\int\limits_{Q}{{{\left(  {(1-\psi_{\delta}^{-})\varphi}\right)
}_{t}}\overline{H_{m}}(U}_{n}{)},\quad\quad\quad D_{2,m}^{n}=\int%
\limits_{Q}H_{m}({U}_{n})A(x,t,\nabla u_{n}).\nabla\left(  {(1-\psi_{\delta
}^{-})\varphi}\right)  ,\\
\\
D_{3,m}^{n}=\int\limits_{Q}H_{m}({U}_{n})(1-\psi_{\delta}^{-})\varphi
d(\widehat{\lambda_{n,0}}+\rho_{n,0}-\eta_{n,0}){{,}}\quad D_{4,m}^{n}%
=\frac{1}{m}\int\limits_{m\leq U\leq2m}{(1-\psi_{\delta}^{-})\varphi
A(x,t,\nabla u_{n}).{\nabla{U}}}_{n}{{,}}\\
\\
D_{5,m}^{n}=-\frac{1}{m}\int\limits_{-2m\leq{U}_{n}\leq-m}(1-\psi_{\delta}%
^{-})\varphi A(x,t,\nabla u_{n}).{{\nabla{U}_{n}}}%
\end{array}
\label{281n}%
\end{equation}
In (\ref{281n}), we go to the limit as $m\rightarrow\infty$. Since $\left\{
\overline{H}_{m}(U_{n})\right\}  $ converges to $U_{n}$ and $\left\{
H_{m}(U_{n})\right\}  $ converges to $1,$ $a.e.$ in $Q,$ and $\left\{  \nabla
H_{m}(U_{n})\right\}  $ converges to $0,$ weakly in $(L^{p}(Q))^{N}$ , we
obtain the relation $D_{1}^{n}+D_{2}^{n}=D_{3}^{n}+D^{n},$ where%
\begin{align*}
D_{1}^{n}  &  =-\int_{Q}{{\left(  {(1-\psi_{\delta}^{-})\varphi}\right)  }%
_{t}U}_{n},\quad D_{2}^{n}=\int_{Q}A(x,t,\nabla u_{n})\nabla\left(
{(1-\psi_{\delta}^{-})\varphi}\right)  ,\quad D_{3}^{n}=\int_{Q}%
{(1-\psi_{\delta}^{-})\varphi d\widehat{\lambda_{n,0}}}\\
D^{n}  &  =\int_{Q}{(1-\psi_{\delta}^{-})\varphi d(\rho_{n,0}-\eta_{n,0}%
)+}\int_{Q}{(1-\psi_{\delta}^{-})\varphi d(({{\rho}}_{n,s}-\eta_{n,s})}%
^{+}-{({{\rho}}_{n,s}-\eta_{n,s})}^{-})\\
&  =\int_{Q}{(1-\psi_{\delta}^{-})\varphi d({{\rho}}_{n}-\eta_{n}).}%
\end{align*}
Clearly, $D{_{i,m}}-D{_{i}^{n}}=\omega(n,m)$ for $i=1,2,3.$ From Lemma
(\ref{april261}) and (\ref{12054})-(\ref{12057}), we obtain $D_{5,m}%
=\omega(n,m,\delta),$ and
\[
{\frac{1}{m}\int\limits_{\left\{  m\leq U<2m\right\}  }\psi_{\delta}%
^{-}\varphi A(x,t,\nabla u).\nabla{U}}=\omega(n,m,\delta),
\]
thus,
\[
D_{4,m}=\frac{1}{m}{\int\limits_{\left\{  m\leq U<2m\right\}  }\varphi
A(x,t,\nabla u).\nabla U}+\omega(n,m,\delta).
\]
Since $\left\vert \int_{Q}{(1-\psi_{\delta}^{-})\varphi d\eta_{n}}\right\vert
\leq{\left\Vert \varphi\right\Vert _{{L^{\infty}}}}\int_{Q}{(1-\psi_{\delta
}^{-})d\eta_{n},}$ it follows that $\int_{Q}{(1-\psi_{\delta}^{-})\varphi
d\eta_{n}}=\omega(n,m,\delta)$ from (\ref{12057})$.$ And $\left\vert \int%
_{Q}{{\psi_{\delta}^{-}\varphi d\rho_{n}}}\right\vert \leq{\left\Vert
\varphi\right\Vert _{{L^{\infty}}}}\int_{Q}{\psi_{\delta}^{-}d\rho_{n},}$
thus, from (\ref{12054}), $\int_{Q}{(1-\psi_{\delta}^{-})\varphi d{{\rho}}%
_{n}}=\int_{Q}{\varphi d\mu_{s}^{+}}+\omega(n,m,\delta).$ Then $D^{n}=\int%
_{Q}{\varphi d\mu_{s}^{+}}+\omega(n,m,\delta).$ Therefore by subtraction, we
get successively
\[
\frac{1}{m}{\int\limits_{\left\{  m\leq U<2m\right\}  }\varphi A(x,t,\nabla
u).\nabla U}=\int_{Q}{\varphi d\mu_{s}^{+}}+\omega(n,m,\delta),
\]%
\begin{equation}
\lim_{m\rightarrow\infty}\frac{1}{m}\int\limits_{\left\{  m\leq U<2m\right\}
}\varphi A(x,t,\nabla u).\nabla U=\int_{Q}{\varphi d\mu_{s}^{+},} \label{en}%
\end{equation}
which proves (\ref{renor2}) when $\varphi\in C_{c}^{\infty}(Q).$ Next assume
only $\varphi\in C^{\infty}(\overline{Q})$. Then
\[%
\begin{array}
[c]{l}%
\lim_{m\rightarrow\infty}\frac{1}{m}{\int\limits_{\left\{  m\leq U<2m\right\}
}\varphi A(x,t,\nabla u).\nabla U}\\
\\
=\lim_{m\rightarrow\infty}\frac{1}{m}{\int\limits_{\left\{  m\leq
U<2m\right\}  }\varphi\psi_{\delta}^{+}A(x,t,\nabla u)\nabla U}+\lim
_{m\rightarrow\infty}\frac{1}{m}{\int\limits_{\left\{  m\leq U<2m\right\}
}\varphi(1-\psi_{\delta}^{+})A(x,t,\nabla u).\nabla U}\\
\\
=\int_{Q}{\varphi\psi_{\delta}^{+}d\mu_{s}^{+}}+\lim_{m\rightarrow\infty}%
\frac{1}{m}{\int\limits_{\left\{  m\leq U<2m\right\}  }\varphi(1-\psi_{\delta
}^{+})A(x,t,\nabla u).\nabla U}=\int_{Q}{\varphi d\mu_{s}^{+}}+D,
\end{array}
\]
where
\[
D=\int_{Q}{\varphi(1-\psi_{\delta}^{+})d\mu_{s}^{+}}+\lim_{n\rightarrow\infty
}\frac{1}{m}{\int\limits_{\left\{  m\leq U<2m\right\}  }\varphi(1-\psi
_{\delta}^{+})A(x,t,\nabla u).\nabla U=\;}\omega(\delta).
\]
Therefore, (\ref{en}) still holds for $\varphi\in C^{\infty}(\overline{Q}),$
and we deduce (\ref{renor2}) by density, and similarly, (\ref{renor3}). This
completes the proof of Theorem \ref{sta}.
\end{proof}

\section{Approximations of measures\label{prox}}

Corollary \ref{051120131} is a direct consequence of Theorem \ref{sta} and the
following approximation property:

\begin{proposition}
\label{4bhatt}Let $\mu=\mu_{0}+\mu_{s}\in\mathcal{M}_{b}^{+}(Q)$ with $\mu
_{0}\in\mathcal{M}_{0}^{+}(Q)$ and $\mu_{s}\in\mathcal{M}_{s}^{+}(Q).$

\noindent(i) Then, we can find a decomposition $\mu_{0}=(f,g,h)$ with $f\in
L^{1}(Q),g\in(L^{p^{\prime}}(Q))^{N},h\in X$ such that
\begin{equation}
||f||_{1,Q}+\left\Vert g\right\Vert _{p^{\prime},Q}+||h||_{X}+\mu_{s}%
(\Omega)\leq2\mu(Q)\label{4bhdeco}%
\end{equation}
(ii) Furthermore, there exists sequences of measures $\mu_{0,n}=(f_{n}%
,g_{n},h_{n}),\mu_{s,n}$ such that $f_{n},g_{n},h_{n}\in C_{c}^{\infty}(Q)$
strongly converge to $f,g,h$ in $L^{1}(Q),(L^{p^{\prime}}(Q))^{N}$ and $X$
respectively, and $\mu_{s,n}\in(C_{c}^{\infty}(Q))^{+}$ converges to $\mu_{s}$
and $\mu_{n}:=\mu_{0,n}+\mu_{s,n}$ converges to $\mu$ in the narrow topology,
and satisfying $|\mu_{n}|(Q)\leq\mu(Q),$
\begin{equation}
||f_{n}||_{1,Q}+\left\Vert g_{n}\right\Vert _{p^{\prime},Q}+||h_{n}||_{X}%
+\mu_{s,n}(Q)\leq2\mu(Q).\label{4bhdecn}%
\end{equation}

\end{proposition}

\begin{proof}
(i) Step 1. Case where $\mu$ has a compact support in $Q.$ By \cite{DrPoPr},
we can find a decomposition $\mu_{0}=(f,g,h)$ with $f,g,h$ have a compact
support in $Q.$ Let $\left\{  \varphi_{n}\right\}  $ be sequence of mollifiers
in $\mathbb{R}^{N+1}$. Then $\mu_{0,n}=\varphi_{n}\ast\mu_{0}\in C_{c}%
^{\infty}(Q)$ for $n$ large enough. We see that $\mu_{0,n}(Q)=\mu_{0}(Q)$ and
$\mu_{0,n}$ admits the decomposition $\mu_{0,n}=(f_{n},g_{n},h_{n}%
)=(\varphi_{n}\ast f,\varphi_{n}\ast g,\varphi_{n}\ast h)$. Since $\left\{
f_{n}\right\}  ,\left\{  g_{n}\right\}  ,\left\{  h_{n}\right\}  $ strongly
converge to $f,g,h$ in $L^{1}(Q),(L^{p^{\prime}}(Q))^{N}$ and $X$
respectively, we have for $n_{0}$ large enough,
\[
||f-f_{n_{0}}||_{1,Q}+||g-g_{n_{0}}||_{p^{\prime},Q}+||h-h_{n_{0}}%
||_{L^{p}((0,T);W_{0}^{1,p}(\Omega))}\leq\frac{1}{2}\mu_{0}(Q).
\]
Then we obtain a decomposition $\mu=(\hat{f},\hat{g},\hat{h})=(\mu_{n_{0}%
}+f-f_{n_{0}},g-g_{n_{0}},h-h_{n_{0}}),$ such that
\begin{equation}
||\hat{f}||_{1,Q}+||\hat{g}||_{p^{\prime},Q}+||\hat{h}||_{X}+\mu_{s}%
(Q)\leq\frac{3}{2}\mu(Q)\label{4bhcomp}%
\end{equation}
Step 2. General case. Let $\{\theta_{n}\}$ be a nonnegative, nondecreasing
sequence in $C_{c}^{\infty}(Q)$ which converges to $1,$ $a.e.$ in $Q$. Set
${\tilde{\mu}_{0}}={\theta_{0}\mu,}$ and ${\tilde{\mu}_{n}}=(\theta_{n}%
-\theta_{n-1})\mu,$ for any $n\geq1$. Since $\tilde{\mu}_{n}=\tilde{\mu}%
_{0,n}+\tilde{\mu}_{s,n}\in\mathcal{M}_{0}(Q)\cap\mathcal{M}_{b}^{+}(Q)$ has
compact support with $\tilde{\mu}_{0,n}\in\mathcal{M}_{0}(Q),\tilde{\mu}%
_{s,n}\in\mathcal{M}_{s}(Q)$, by Step 1, we can find a decomposition
$\tilde{\mu}_{0,n}=(\tilde{f}_{n},\tilde{g}_{n},\tilde{h}_{n})$ such that
\[
||\tilde{f}_{n}||_{1,Q}+\left\Vert \tilde{g}_{n}\right\Vert _{p^{\prime}%
,Q}+||\tilde{h}_{n}||_{X}+\tilde{\mu}_{s,n}(\Omega)\leq\frac{3}{2}\tilde{\mu
}_{n}(Q).
\]
Let $\overline{f}_{n}=\sum\limits_{k=0}^{n}{\tilde{f}}_{k}$, $\overline{g}%
_{n}=\sum\limits_{k=0}^{n}\tilde{g}_{k}$, $\bar{h}_{n}=\sum\limits_{k=0}%
^{n}\tilde{h}_{k}$ and $\bar{\mu}_{s,n}=\sum_{k=0}^{n}\tilde{\mu}_{s,k}$.
Clearly, $\theta_{n}\mu_{0}=(\overline{f}_{n},\overline{g}_{n},\bar{h}_{n}),$
$\theta_{n}\mu_{s}=\bar{\mu}_{s,n}$ and $\left\{  \overline{f}_{n}\right\}
,\left\{  \overline{g}_{n}\right\}  ,\left\{  \bar{h}_{n}\right\}  $ and
$\left\{  \bar{\mu}_{s,n}\right\}  $ converge strongly to some $f,g,h,$ and
$\mu_{s}$ respectively in $L^{1}(Q)$,$(L^{p^{\prime}}(Q))^{N}$, $X$  and
$\mathcal{M}_{b}^{+}(Q),$ and
\[
||\overline{f}_{n}||_{1,Q}+||\overline{g}_{n}||_{p^{\prime},Q}+||\bar{h}%
_{n}||_{X}+\bar{\mu}_{s,n}(Q)\leq\frac{3}{2}\mu(Q).
\]
Therefore, $\mu_{0}=(f,g,h)$, and (\ref{4bhdeco}) holds.\medskip

(ii) We take a sequence $\{m_{n}\}$ in $\mathbb{N}$ such that $f_{n}%
=\varphi_{m_{n}}\ast\overline{f}_{n},g_{n}=\varphi_{m_{n}}\ast\overline{g}%
_{n},h_{n}=\varphi_{m_{n}}\ast\bar{h}_{n},\varphi_{m_{n}}\ast\bar{\mu}%
_{s,n}\in(C_{c}^{\infty}(Q))^{+},$ $\int_{Q}\varphi_{m_{n}}\ast\bar{\mu}%
_{s,n}dxdt=\bar{\mu}_{s,n}(Q)$ and
\[
||f_{n}-\overline{f}_{n}||_{1,Q}+||g_{n}-\overline{g}_{n}||_{p^{\prime}%
,Q}+||h_{n}-\bar{h}_{n}||_{X}\leq\frac{1}{n+2}\mu(Q).
\]
Let $\mu_{0,n}=\varphi_{m_{n}}\ast(\theta_{n}\mu_{0})=(f_{n},g_{n},h_{n})$,
$\mu_{s,n}=\varphi_{m_{n}}\ast\bar{\mu}_{s,n}$ and $\mu_{n}=\mu_{0,n}%
+\mu_{s,n}$. Therefore, $\left\{  f_{n}\right\}  ,\left\{  g_{n}\right\}
,\left\{  h_{n}\right\}  $ strongly converge to $f,g,h$ in $L^{1}%
(Q),(L^{p^{\prime}}(Q))^{N}$ and $X$ respectively. And (\ref{4bhdecn}) holds.
Furthermore, $\left\{  \mu_{s,n}\right\}  ,\left\{  \mu_{n}\right\}  $
converge to $\mu_{s},\mu$ in the weak topology of measures$,$ and $\mu
_{s,n}(Q)=\int_{Q}\theta_{n}d\mu_{s},\mu_{n}(Q)=\int_{Q}\theta_{n}d\mu$
converges to $\mu_{s}(Q),\mu(Q)$, thus $\left\{  \mu_{s,n}\right\}  ,\left\{
\mu_{n}\right\}  $ converges to $\mu_{s},\mu$ in the narrow topology and
$|\mu_{n}|(Q)\leq\mu(Q)$.\medskip
\end{proof}

Observe that part (i) of Proposition \ref{4bhatt} was used in \cite{Pe08},
even if there was no explicit proof. Otherwise part (ii) is a \textit{ key
point} for finding applications to the stability Theorem. Note also a very
useful consequence for approximations by \textit{nondecreasing} sequences:

\begin{proposition}
\label{P5} Let $\mu\in\mathcal{M}_{b}^{+}(Q)$ and $\varepsilon>0$. Let
$\left\{  \mu_{n}\right\}  $ be a nondecreasing sequence in $\mathcal{M}%
_{b}^{+}(Q)$ converging to $\mu$ in $\mathcal{M}_{b}(Q)$. Then, there exist
$f_{n},f\in L^{1}(Q)$, $g_{n},g\in(L^{p^{\prime}}(Q))^{N}$ and $h_{n},h\in X,$
$\mu_{n,s},\mu_{s}\in\mathcal{M}_{s}^{+}(Q)$ such that
\[
\mu=f-\operatorname{div}g+h_{t}+\mu_{s},\qquad\mu_{n}=f_{n}-\operatorname{div}%
g_{n}+(h_{n})_{t}+\mu_{n,s},
\]
and $\left\{  f_{n}\right\}  ,\left\{  g_{n}\right\}  ,\left\{  h_{n}\right\}
$ strongly converge to $f,g,h$ in $L^{1}(Q),(L^{p^{\prime}}(Q))^{N}$ and $X$
respectively, and $\left\{  \mu_{n,s}\right\}  $ converges to $\mu_{s}$
(strongly) in $\mathcal{M}_{b}(Q)$ and
\begin{equation}
||f_{n}||_{1,Q}+||g_{n}||_{p^{\prime},Q}+||h_{n}||_{X}+\mu_{n,s}(\Omega
)\leq2\mu(Q). \label{4bh2504}%
\end{equation}

\end{proposition}

\begin{proof}
Since $\left\{  \mu_{n}\right\}  $ is nondecreasing, then $\left\{  \mu
_{n,0}\right\}  $, $\left\{  \mu_{n,s}\right\}  $ are nondecreasing too.
Clearly, $\left\Vert {\mu-{\mu_{n}}}\right\Vert _{\mathcal{M}_{b}%
(Q)}=\left\Vert {{\mu_{0}}-{\mu_{n,0}}}\right\Vert _{\mathcal{M}_{b}%
(Q)}+\left\Vert {{\mu_{s}}-{\mu_{n,s}}}\right\Vert _{\mathcal{M}_{b}(Q)}$.
Hence, $\left\{  \mu_{n,s}\right\}  $ converges to $\mu_{s}$ and $\left\{
\mu_{n,0}\right\}  $ converges to ${{\mu_{0}}}$ (strongly) in $\mathcal{M}%
_{b}(Q)$. Set ${\widetilde{\mu}_{0,0}}={\mu_{0,0},}$ and ${\widetilde{\mu
}_{n,0}}={\mu_{n,0}}-{\mu_{n-1,0}}$ for any $n\geq1$. By Proposition
\ref{4bhatt}, (i), we can find $\tilde{f}_{n}\in L^{1}(Q)$, $\tilde{g}_{n}%
\in(L^{p^{\prime}}(Q))^{N}$ and $\tilde{h}_{n}\in X$ such that $\tilde{\mu
}_{n,0}=(\tilde{f}_{n},\tilde{g}_{n},\tilde{h}_{n})$ and
\[
||\tilde{f}_{n}||_{1,Q}+||\tilde{g}_{n}||_{p^{\prime},Q}+||\tilde{h}_{n}%
||_{X}\leq2\tilde{\mu}_{n,0}(Q)
\]

\noindent Let $f_{n}=\sum\limits_{k=0}^{n}{\tilde{f}}_{k}$, $G_{n}%
=\sum\limits_{k=0}^{n}\tilde{g}_{k}$ and $h_{n}=\sum\limits_{k=0}^{n}\tilde
{h}_{k}$. Clearly, $\mu_{n,0}=(f_{n},g_{n},h_{n})$ and the convergence
properties hold with (\ref{4bh2504}), since
\[
||f_{n}||_{1,Q}+||g_{n}||_{p^{\prime},Q}+||h_{n}||_{X}\leq2\mu_{0}(Q).
\]

\end{proof}

\end{document}